\documentclass[11pt,reqno]{amsproc}

\usepackage{amsmath,amsfonts,amssymb,amsthm}
\usepackage[abbrev,lite,nobysame]{amsrefs}
\usepackage{mathrsfs}
\usepackage{graphics,graphicx}
\usepackage[usenames,dvipsnames]{color}
\usepackage{times}
\usepackage{bbm}
\usepackage[margin=1in]{geometry}
\usepackage[colorlinks=true, pdfstartview=FitV, linkcolor=BrickRed,citecolor=black, urlcolor=black]{hyperref}

\usepackage{mathtools}
\mathtoolsset{showonlyrefs}

\newtheorem{theorem}{Theorem}
\newtheorem{proposition}{Proposition}[section]
\newtheorem{lemma}[proposition]{Lemma}

\theoremstyle{definition}

\newtheorem{remark}[proposition]{Remark}

\numberwithin{equation}{section}

\makeatletter

\def\@tocline#1#2#3#4#5#6#7{\relax
  \ifnum #1>\c@tocdepth 
  \else
    \par \addpenalty\@secpenalty\addvspace{#2}%
    \begingroup \hyphenpenalty\@M
    \@ifempty{#4}{%
      \@tempdima\csname r@tocindent\number#1\endcsname\relax
    }{%
      \@tempdima#4\relax
    }%
    \parindent\z@ \leftskip#3\relax \advance\leftskip\@tempdima\relax
    \rightskip\@pnumwidth plus4em \parfillskip-\@pnumwidth
    #5\leavevmode\hskip-\@tempdima
      \ifcase #1
       \or\or \hskip 1em \or \hskip 2em \else \hskip 3em \fi%
      #6\nobreak\relax
    \dotfill\hbox to\@pnumwidth{\@tocpagenum{#7}}\par
    \nobreak
    \endgroup
  \fi}
\makeatother

\newcommand\eps{\varepsilon}
\newcommand\e{{\rm e}}
\newcommand\dd{{\rm d}}
\newcommand\ddt{{\frac{\dd}{\dd t}}}

\newcommand\de{{\partial}}

\newcommand{\NN}{\mathbb{N}}

\newcommand\RR {{\mathbb R}}
\newcommand\DD {{\mathbb D}}
\def\SS {{\mathbb S}}

\newcommand\bu{{\boldsymbol u}}

\newcommand\cO{{\mathcal O}}
\newcommand\cP{{\mathcal P}}

\newcommand\cS{{\mathcal S}}

\newcommand\sfA{{\mathsf A}}

\newcommand\sfE{{\mathsf E}}
\newcommand\sfF{{\mathsf F}}
\newcommand\sfG{{\mathsf G}}

\newcommand\sfK{{\mathsf K}}
\newcommand\sfL{{\mathsf L}}

\newcommand\sfS{{\mathsf S}}

\newcommand\sfa{{\mathsf a}}

\newcommand\sfe{{\mathsf e}}

\newcommand\sfm{{\mathsf m}}

\usepackage{mathtools}
\usepackage{bbm}

\usepackage{enumitem}

 \newcommand\ind{{\mathbbm 1}}

 \newcommand\LL{{L}}
 
\begin{document}

\title[Entropy maximization in the $2d$ Euler equations]{\vspace*{-3cm} Entropy maximization in the two-dimensional Euler equations} 
\author[M. Coti Zelati]{Michele Coti Zelati}
\address{Department of Mathematics, Imperial College London, London, SW7 2AZ, UK}
\email{m.coti-zelati@imperial.ac.uk}

\author[M.G. Delgadino]{Matias G. Delgadino}
\address{Department of Mathematics, The University of Texas at Austin, TX 78712-1202, USA}
\email{matias.delgadino@math.utexas.edu}

\subjclass[2020]{35Q31, 37K58, 49Q22}

\keywords{Euler equations, entropy maximization, statistical hydrodynamics, rearrangement inequalities, optimal transport}

\begin{abstract}
We consider variational problem related to entropy maximization in the two-dimensional Euler equations, in order  to investigate the long-time dynamics of solutions with bounded vorticity. 
Using variations on the classical min-max principle and borrowing ideas from optimal transportation and quantitative rearrangement inequalities, we prove results on the structure of entropy
maximizers arising in the investigation of the long-time behavior of vortex patches. We further show that the same techniques apply in the study of stability of the canonical Gibbs measure associated to a system of point vortices.
\end{abstract}

\maketitle

\tableofcontents

\section{Long-time dynamics in two-dimensional perfect fluids}\label{sec:intro}
The Euler equations describing the motion of an inviscid and incompressible fluid in a two-dimensional regular simply connected domain $M\subset\RR^2$ read
\begin{equation}\label{eq:Euler}\tag{E}
\begin{cases}
\de_t\omega+\bu\cdot \nabla \omega=0,\\
\omega|_{t=0}=\omega^{in},
\end{cases}
\end{equation}
where $\omega(t,x):\RR\times M\to \RR$ is the vorticity and $\bu(t,x):\RR\times M\to \RR^2$ is the divergence-free velocity field, related to $\omega$
through the \emph{Biot-Savart} law
\begin{equation}\label{eq:stream}
\bu=\nabla^\perp \psi=(-\de_{x_2}\psi,\de_{x_1}\psi), \qquad
\begin{cases}
\Delta\psi=\omega,\quad &\text{in } M,\\
\psi=0,\quad &\text{on }\de M.
\end{cases}
\end{equation}
In short, $\bu=\nabla^\perp\Delta^{-1}\omega$. The transport nature of equation \eqref{eq:Euler} translates into the representation of solution
via the method of characteristics
\begin{equation}
\omega(t,x) =\omega^{in} \circ \Phi^{-1}_t (x),
\end{equation}
where
\begin{equation}
\ddt \Phi_t(x) =\bu (t,\Phi_t(x)),\qquad \Phi_0(x)=x,
\end{equation}
is the Lagrangian flow. Thanks to Yudovich theory \cite{Yudovich63}, the Euler equations can be seen as a well-posed, weak-$*$ continuous dynamical system on the 
compact metric space given by the unit ball in $L^\infty$
\begin{equation}
X:= \left\{\omega\in L^\infty(M): \|\omega\|_{L^\infty}\leq 1\right\},
\end{equation}
endowed with the weak-$*$ topology (see \cite{Nguyen22} for a recent proof of this fact).  It is then natural to ask what is the generic long-time picture of solutions
to \eqref{eq:Euler}. A central conjecture due to V. Šverák \cite{Sverak12} posits that generic initial data give rise to solutions whose orbits are not precompact in $L^2$. While
this conjecture, in its generality, remains currently out of reach, the recent \emph{inviscid damping} results \cites{BM15,MZ20,IJ22,IJ20,IJnon20} validate it in certain perturbative regimes.

A related point of view revolves around the idea that the velocity causes a cascade towards high frequencies that averages (i.e. \emph{mixes}) the vorticity in infinite time. The Euler equations \eqref{eq:Euler} preserve physically relevant quantities, hence said frequency cascade is constrained to be consistent with many conservation laws. These are the \emph{kinetic energy}
\begin{equation}
\sfE(\omega)=\frac12\int_M | \nabla^\perp \Delta^{-1}\omega(x) |^2\dd x=\frac12\int_M | \bu(x) |^2\dd x,
\end{equation}
the \emph{circulation}
\begin{equation}
\sfK(\omega)=\int_{\de M} \bu \cdot\dd\ell = \int_M \omega(x) \dd x,
\end{equation}
and the \emph{Casimirs}
\begin{equation}
\sfS_f(\omega)=\frac{1}{|M|}\int_M f(\omega(x))\dd x, \qquad \text{for any continuous } f:\RR\to \RR.
\end{equation}
Along any sequence of times tending to infinity, weak compactness implies the existence of subsequential limit points $\omega_\infty$ for the dynamics.
 While $\sfE$ and $\sfK$  are continuous in the weak-$*$ topology, and hence  
 \begin{equation}\label{eq:conserved}
 	(\sfE,\sfK)(\omega^{in})=(\sfE,\sfK)(\omega_\infty),
 \end{equation}
the Casimirs may lose information at infinite time. In particular, if $\omega(t_j) \xrightharpoonup{*}\omega_\infty$, along a sequence of time $t_j\to\infty$,
then only upper-semicontinuity can be deduced, namely 
\begin{equation}
\sfS_f(\omega^{in})=\limsup_{j\to\infty}\sfS_f(\omega(t_j))\leq\sfS_f(\omega_\infty), \qquad \text{for any continuous concave } f:\RR\to \RR.
\end{equation}
A strict inequality above is associated to mixing and is often observed in the long-time limit of the two-dimensional Euler equations. 

To give the above observations a robust mathematical framework, one can account for the Euler evolution and its long-time limits by considering 
the weakly-$*$ closed set
\begin{equation}\label{eq:Osets}
\cO_{in} = \overline{\left\{\omega^{in}\circ \Phi \ |\ \Phi:M\to M \text{ is an area preserving diffeomorphism} \right\}}^*,
\end{equation}
which can be seen as the \emph{orbit} of the natural action of the volume-preserving
diffeomorphism group on the vorticity field $\omega^{in}$.
Although this set may strictly contain the $\Omega$-limit set of $\omega^{in}$, we can get close to the dynamics of $2d$ Euler
by intersecting $\cO_{in}$ with the various conservation laws \eqref{eq:conserved}. This approach provides the basis of Shnirelman’s maximal mixing
theory \cite{Shnirelman93}, explored and revisited recently in \cite{DD22}.

\subsection{A statistical mechanics perspective}\label{sub:statmech}
For specific choices of $f$ (e.g. $f(\omega)=-\omega\log\omega$),  $\sfS_f$ can be seen as a measure of \emph{entropy}, which is a measure of the number of possible configurations at the microscopic level that leads to the observable macrostate. The \textit{second law of thermodynamics} states that the \textit{entropy} of an isolated system will never decrease, but will instead tend to increase over time until it reaches a maximum value at equilibrium. 

In his seminal 1949 work \cite{Onsager49}, L. Onsager argued, under certain ergodicity assumptions, that Euler flows originating from point vortices should relax to vorticities that maximize the Boltzmann entropy 
\begin{equation}\label{eq:boltz}
\sfS(\omega):=-\frac{1}{|M|}\int_M  \omega(x) \log\omega(x) \dd x,
\end{equation}
subject to all conservation laws, in analogy with equilibrium statistical mechanics.  The field has seen tremendous growth since then, with the development
of \emph{statistical hydrodynamics} theories corresponding to variational problems of the type
\begin{equation}\label{eq:maxprobgen}
\text{maximize } \sfS_f(\omega)\text{ subject to } \omega\in \cO_{in} \quad \mbox{ and }\quad \sfE(\omega) =\sfE(\omega^{in}),
\end{equation} 
for suitable choices of $f$, see the reviews \cites{Robert95,BV12}. A rigorous picture for point vortices has been established in the seminal articles \cites{CLMP92,CLMP95,ES93,Kiessling93}. 
In this article we center in the choice of the Boltzmann entropy \eqref{eq:boltz}, but  other choices of $f$  are also physically relevant.

In the canonical formalism of statistical mechanics, the maximal entropy functional is computed via a variational problem over a microcanonical ensemble as
\begin{equation}\label{eq:maxentropy}
\cS(\sfe)=\max \left\{\sfS_f(\omega): \omega\in \cO_{in},\ \sfE(\omega) =\sfe  \right\}, \qquad 
\end{equation}
and is expected to be concave with respect to the energy level $\sfe$.  However, even in the simple scenario of a single vortex patch in a disk,
 this seems to be an open question \cite{Sverak12}.
 
One of the aims of this article is to give
a general strategy to show the concavity of $\cS$, via a variation of the classical min-max principle. The energy constraint in \eqref{eq:maxentropy} can indeed be  rewritten as 
\begin{equation}\label{eq:Srelaxed}
\cS(\sfe)=\max_{\omega\in \cO_{in}} \left\{ \sfS_f(\omega) +\inf_{\beta\in \RR} \beta \left(\sfe-\sfE(\omega)\right)\right\},
\end{equation}
since
\begin{equation}
\inf_{\beta\in \RR}\beta\left(\sfe-\sfE(\omega)\right)=\begin{cases}
0,& \mbox{if $\sfE(\omega)=\sfe$},\\
-\infty, &\mbox{otherwise}.
\end{cases}
\end{equation}
By formally commuting the $\max$ and the $\inf$ in \eqref{eq:Srelaxed}, also known as the min-max principle, we are able to rewrite
\begin{equation}\label{eq:Sminmax}
\cS(\sfe)=\inf_{\beta\in \RR}\left\{ \beta \sfe +\max_{\omega\in \cO_{in}} \left( \sfS_f(\omega) - \beta \sfE(\omega)\right)\right\}=
\inf_{\beta\in \RR}\left\{ \beta \sfe +\max_{\omega\in \cO_{in}} \sfF_{\beta}(\omega) \right\}.
\end{equation}
in terms of the associated free energy 
\begin{equation}\label{eq:freeenergy}
 \sfF_{\beta}(\omega):= \sfS_f(\omega) -\beta \sfE(\omega).
\end{equation}
If the min-max principle applies, it follows immediately that $\cS$ is concave as it is now written as the infimum over affine functions of $\sfe$. The rigorous justification of the min-max theorem is included in  Appendix \ref{app:minmax}, whose main requirement is the uniqueness of maximizers of $ \sfF_{\beta}$ in $\cO_{in}$ for every  $\beta\in\RR$. 
This discussion leads to the following conditional result.

\begin{theorem}[Uniqueness implies concavity]\label{thm:main}
Assume that $\omega^{in}=\ind_A$, for some $A\subset M$, and suppose that $f\in C([0,\infty))\cap C^1((0,\infty))$ is strictly concave with $f'(z)\to -\infty$ as $z\to 0^+$. If for any achievable energy level the constrained entropy maximization problem \eqref{eq:maxprobgen} has a unique maximizer, then the maximal entropy functional $\cS$ in \eqref{eq:maxentropy} is strictly concave.
\end{theorem}
\begin{remark}
    The uniqueness of the entropy maximization problem \eqref{eq:maxprobgen} can be relaxed to allow uniqueness up to transformations that preserve the entropy and the energy, see Remark~\ref{rem:app}. For instance, in the case of radially symmetric domains, we can for instance allow for uniqueness up to rotations.
\end{remark}

Showing the uniqueness property required in the above Theorem \ref{thm:main} in general settings is an interesting and challenging open question. In statistical physics, non-uniqueness of the equilibrium state is directly related to phase transitions in the associated system \cites{HO11,MR4271956,MR4604897}. For interacting particle \cites{CGPS20,CP10} and spin glass systems, several toy models exhibiting discontinuous phase transitions exist. In the case of $2d$ Euler, it could be expected that for non-trivial geometries, a discontinuous phase transitions takes place (specifically, the relevant type are first order or discontinuous phase transitions).

While in the variational problem  it is just a Lagrange multiplier,  
$\beta$ plays the role of the inverse temperature in the language of statistical mechanics. More specifically, if there exists a unique $\beta=\beta(\sfe)$ that attains the infimum in \eqref{eq:Sminmax},
we obtain the classical statistical mechanics identity (Classius Law)
\begin{equation}
\beta=\frac{\dd \cS}{\dd\sfe}. 
\end{equation}
That is to say, the derivative of the maximal entropy with respect to the energy is the inverse temperature of the system. In particular,
$\beta(\sfe)$ is strictly decreasing in $\sfe$, and hence we can write the energy associated to a given inverse temperature level $\beta$ (see Figure \ref{fig:Sfunction}).

\begin{figure}[ht]
    \includegraphics[width=0.5\linewidth]{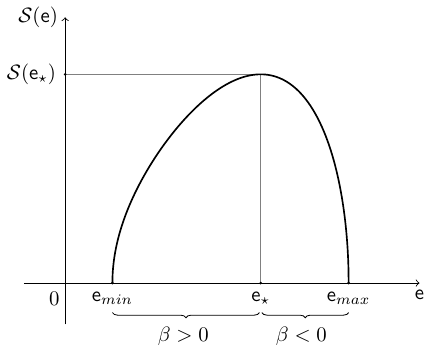}
      \caption{The function $\sfe\mapsto\cS(\sfe)$. The three cases $\beta\to -\infty$, $\beta=0$ and $\beta\to\infty$ are associated with the energy levels $\sfe_{max}$, $\sfe_\star$ and $\sfe_{min}$, respectively.}
  \label{fig:Sfunction}
\end{figure}

\subsection{Main results for vortex patches}\label{sub:mainresult}
One of the purposes of this article is to show the uniqueness of maximizers in \eqref{eq:maxprobgen} in the case when $\omega^{in}$ is a vortex patch on the disk, by exploiting the the theory of radial rearrangements. 
We therefore consider $M=\DD$, the unit disk centered at the origin, 
and  $\omega^{in}$  the indicator function of a set $A\subset \DD$, i.e.
\begin{equation}\label{eq:vortexpatch}
\omega^{in}(x)=\ind_A(x), \qquad \frac{|A|}{|\DD|} =\sfm \in(0,1).
\end{equation}
In this context, the corresponding set in \eqref{eq:Osets} takes the particularly amenable form \cite{DD22}
\begin{equation}\label{eq:Ochar}
\cO:=\left\{\omega\in L^\infty: 0\leq \omega\leq 1, \ \frac{1}{|\DD|}\int_\DD \omega(x)\dd x =\sfm\right\}.
\end{equation}
Since $\cO$ is weak-$*$ compact, the weak continuity of $\sfE$ implies the existence of a maximum and a minimum, 
denoted $\sfe_{max}$ and $\sfe_{min}$, respectively. 
Also, since $0\notin \cO$ and $\sfE(\omega)=0$ if and only if $\omega=0$, it necessarily holds that $\sfe_{min}>0$.
In fact, $\sfe_{max}$ and $\sfe_{min}$ and corresponding vorticities can be computed explicitly, 
see Lemma \ref{lem:energymaxmin}. 
Given $\sfe\in[\sfe_{min},\sfe_{max}]$, we are interested in the  maximization problem
\begin{equation}\label{pb:maxpb}
\text{maximize } \sfS(\omega)\text{ subject to } \omega\in \cO\mbox{ and 
}\sfE(\omega) =\sfe,
\end{equation} 
where $\sfS$ denotes the Boltzmann  entropy 
\begin{equation}\label{eq:boltzDisc}
\sfS(\omega):=-\frac{1}{|\DD|}\int_\DD  \omega(x) \log\omega(x) \dd x.
\end{equation}

\begin{remark}[On the choice of entropy]
Our analysis does not heavily rely on the specific form \eqref{eq:boltzDisc}, although for the patch problem, several other entropies can be considered, consistent with classical theories of statistical hydrodynamics. For instance, in the Robert–Sommeria–Miller theory \cites{Robert90,miller1990statistical,robert1991statistical,Robert91}, the choice of $\sfS$ should be dictated by form of the initial datum: assume that $\DD$ is partitioned into a disjoint union of sets $\{A_\ell\}_{\ell=1}^N$, and $\omega^{in}=\sum_\ell a_\ell \ind_{A_\ell}$  with $a_\ell\in [0,1]$. Then  one can define the entropy (generated by $\omega^{in}$) as
\begin{equation}\label{eq:naturalentropy}
\sfS_{rsm}(\omega):=\sup\left\{-\frac{1}{|\DD|}\int_\DD \sum_\ell\rho_\ell(x) \log\rho_\ell(x) \dd x: \ \omega=\sum_\ell a_\ell \rho_\ell,\ 0\leq \rho_\ell\leq 1, \ \sum_\ell  \rho_\ell=1\right\}. 
\end{equation}
The case of a vortex patch \eqref{eq:vortexpatch} leads to explicit computations: since $A_1=A$, $A_2=\DD\setminus A$, $a_1=1$ and $a_2=0$, the only possible choice in \eqref{eq:naturalentropy} is when $\rho_1=\omega$ and $\rho_2=1-\omega$, leading to 
\begin{equation}
    \sfS_{rsm}(\omega):=-\frac{1}{|\DD|}\int_\DD \left[ \omega(x) \log\omega(x)+(1-\omega(x)) \log(1-\omega(x))\right] \dd x,
\end{equation}
Another possibility, introduced by Turkington \cite{Turkington99}, is to consider a different maximization problem compared to \eqref{eq:naturalentropy}, that, in the case of a vortex patch, allows vorticity to mix on the whole range of small scales $a\in [a_2,a_1]=[0,1]$. This leads to the entropy
\begin{equation}\label{eq:Turkington}
\sfS_t(\omega):=\frac{1}{|\DD|}\int_\DD f_{t}(\omega(x))\dd x,
\end{equation}
with
\begin{equation}
f_{t}(\omega)=\sup\left\{-\int_0^1\rho(y)\log\rho(y)\dd y:\ 0\leq \rho\leq 1,\ \int_0^1 \rho(y)\dd y=1, \ \int_0^1 y\rho(y)\dd y=\omega\right\}.
\end{equation}
\end{remark}
The first main result of this article is the following characterization of the the maximal entropy functional, in complete analogy with classical statistical mechanics.
\begin{theorem}\label{thm:concavity}
The function $\cS:[\sfe_{min},\sfe_{max}]\to[0,\infty)$ defined as
\begin{equation}\label{def:S}
\cS(\sfe)=\max \left\{\sfS(\omega): \omega\in \cO,\ \sfE(\omega) =\sfe  \right\}.
\end{equation}
is strictly concave. 
Moreover,     
\begin{enumerate}[label=(\alph*), ref=(\alph*)]
        \item\label{item:maximum} $\cS$ has a unique, strictly positive global maximum at  $ \sfe_\star=\pi\sfm^2/16$;
        \item\label{item:values}  $\cS(\sfe_{min})=\cS(\sfe_{max})=0$;    
        \item\label{item:monotone} $\cS$ is increasing on $[\sfe_{min},\sfe_\star]$ and decreasing on $[\sfe_\star,\sfe_{max}]$.
\end{enumerate}
\end{theorem}
\begin{remark}
Theorem \ref{thm:concavity}\ref{item:monotone} is a consequence of concavity and Theorem \ref{thm:concavity}\ref{item:maximum}-\ref{item:values}.
\end{remark}

In light of Theorem \ref{thm:main}, the rigorous justification of the min-max principle in Appendix \ref{app:minmax} needed for Theorem \ref{thm:concavity} requires
the uniqueness of maximizers of $\sfF_\beta$ in $\cO$ for every $\beta\in\RR$. This is the second main result of this article.

\begin{theorem}\label{thm:uniqueness}
For any $\beta\in \RR$, there exists a unique solution $\omega_\beta\in\cO$
of the maximization problem
\begin{equation}
\text{maximize } \sfF_\beta(\omega)\text{ subject to } \omega\in \cO.
\end{equation} 
Moreover,     
\begin{itemize}
        \item if $\beta\ge 0$, then $\omega_\beta$ is radially increasing.
        \item if $\beta\le 0$, then $\omega_\beta$ is radially decreasing.
\end{itemize}
In particular,  $\omega_\beta$  is constant when $\beta=0$.
\end{theorem}

Our uniqueness proof follows closely the developments in uniqueness of steady states for the standard Keller-Segel model \cite{MR2929020}, and their homogeneous variants \cites{CCH21,CCV15}. In broad terms, the strategy is to find suitable interpolation curves between two competitor states, such that the free energy over the curve is convex or monotone. The seminal paper of McCann \cite{McCann} implements this idea with the interpolation curves given by the geodesics of the optimal transportation distance. We also mention the novel interpolation curve in \cite{DYY} for radially decreasing states, which unfortunately is not directly applicable to our setting. Inspired by \cites{MR2929020,CCH21,CCV15}, we first use rearrangement theory \cites{Talenti76,Kesavan06} to reduce the problem to one dimensional radially decreasing profiles. Then we consider the optimal transportation interpolation between a maximizer and a competitor. Finally, employing Jensen's inequality and the Euler-Lagrange equation for the maximizer, we can show the strict monotonicity of the free energy along the interpolation curve. One of the main difference with \cites{CCH21,CCV15} is that we need to deal with the $L^\infty$ constraint, imposed by \eqref{eq:Ochar}.

\subsection{On the conservation of angular momentum}
The constrained maximization problem \eqref{pb:maxpb} takes into account the conservation laws \eqref{eq:conserved}, but neglects the
additional symmetries of the disk that give rise to conservation of \emph{angular momentum}
\begin{equation}\label{eq:angmom}
\sfA(\omega)=-\int_{\DD} \frac12 \left(1-|x|^2\right)\omega(x)\dd x= \int_{\DD} x^\perp\cdot \bu(x) \dd x,
\end{equation}
which is weak-$*$ continuous, like the energy $\sfE$.

The inclusion of this extra constraint changes the picture dramatically. Indeed, entropy maximizers are no longer necessarily radially symmetric, as the following heuristics illustrate. Ignoring for the moment the $L^\infty$ constraint that the set $\cO$ in \eqref{eq:Ochar} imposes, the limiting case of maximal angular momentum $\sfa_{max}=0$ is achieved only for states $\omega$ which are supported on the boundary $\de\DD$. In which case, if $\omega_{rad}$ is radial then $
\omega_{rad}=c\delta_{\partial\DD}
$
and
$\sfE(\omega_{rad})=0$. On the other hand, the non-radial state
$
\omega_{x_0}=\delta_{x_0}
$ 
with $x_0\in\partial\DD$ also has zero angular momentum, and formally has unbounded energy $\sfE(\omega_{x_0})=+\infty$. This extreme situation hints that constraining the angular momentum implies there are energy levels that are not achievable by radial vorticities. A rigorous statement describing this situation is contained in the following result.

\begin{theorem}\label{thm:nonrad}
For $\sfm\in (0,1)$ and $\LL\geq 1$, consider the set
\begin{equation}\label{eq:OcharL}
\cO_\LL:=\left\{\omega\in L^\infty: 0\leq \omega\leq \LL, \ \frac{1}{|\DD|}\int_\DD \omega(x)\dd x =\sfm\right\}.
\end{equation}
For the radial optimization problem, there exists $C>1$ independent of $\LL$, $\sfa$ and $\sfm$ such that
\begin{equation}\label{eq:rad}
\sup \left\{\sfE(\omega): \omega\in\cO_\LL,\ \sfA(\omega)=\sfa,\ \omega\ \mbox{is radial}\right\}\le C\left(\sfm |\sfa|+|\sfa|^2\log\left(\frac{\LL}{|\sfa|}\right)\right).
\end{equation}
For the non-radial case, if 
$L\ge      \frac{4\pi^2\sfm^3}{|\sfa|^2}$, we have the lower bound
\begin{equation}\label{eq:nonrad}
    \sup \left\{\sfE(\omega): \omega\in\cO_\LL,\ \sfA(\omega)=\sfa\right\} 
    \ge \frac{\pi\sfm^2}{4}\log \left(\frac{\LL|\sfa|^2}{64 \pi^2\sfm^3 }\right).
\end{equation}
In particular, there exist $\sfa_*\in (-1/2,0)$ and $Q>2\pi$ depending on $\sfm$, such that if $\sfa\in (\sfa_*,0)$ and $\LL=Q^2\frac{m^3}{|\sfa|^{2}}$, then
$$
\sup \left\{\sfE(\omega): \omega\in\cO_\LL,\ \sfA(\omega)=\sfa,\ \omega\ \mbox{is radial}\right\}<\sup \left\{\sfE(\omega): \omega\in\cO_\LL,\ \sfA(\omega)=\sfa\right\}.
$$
\end{theorem}

The radial bound \eqref{eq:rad} follows by utilizing the formula
\begin{equation}
\sfE(\omega)
=\frac{1}{4\pi} \int_0^1 \frac1r \left|\int_{B_r} \omega(x)\;\dd x\right|^2 \dd r,
\end{equation}
which is valid only for radial functions. The bound follows by estimating the amount of vorticity near the origin, using the $L^\infty$ bound and the angular momentum. The lower bound \eqref{eq:nonrad} follows by calculating the energy of a vortex patch of the form
$$
\omega_{x_0,\LL}=\LL\ind_{B_{\sqrt{\frac{\sfm}{\LL}}}(x_0)}.
$$
The complete proof of Theorem \ref{thm:nonrad} is postponed to Section \ref{sec:angmom}.

We note that the proof Theorem \ref{thm:nonrad} is similar in spirit to \cite{DD22}*{Theorem 2}. The main difference is that the authors of \cite{DD22} consider the case of a periodic channel (hence not simply connected) instead of the disk. In their case, the conserved quantity of interest is the linear momentum instead of the angular momentum.

\subsection{On the stability of Onsager solutions}
The Euler-Lagrange equations associated to the variational problem \eqref{pb:maxpb} resemble those appearing in the context of mean-field limits of point-vortices studied in \cites{CLMP92,CLMP95,ES93,Kiessling93}. 
Specifically, in the setting of the unit disk  and for $\beta\in (-8\pi,\infty)$, there exists a unique radial solution
to the mean field equation
\begin{equation}\label{eq:MFequation}
\omega_\beta=\frac{\e^{\beta\psi_\beta}}{Z_\beta}\qquad\mbox{and}\qquad
\begin{cases}
\Delta \psi_\beta=\omega_\beta, &\mbox{in $\DD$}\\
\psi_\beta=0, &\mbox{on $\partial\DD$},
\end{cases}
\end{equation}
where
\begin{equation}
Z_\beta = \int_{\DD} \e^{\beta\psi_\beta(x)} \dd x,
\end{equation}
which is given by in radial variables by
\begin{equation}\label{eq:omegabeta}
\omega_\beta(r)=\frac{1-A(\beta)}{\pi}\frac{1}{(1-A(\beta)r^2)^2},\qquad\mbox{with}\qquad A(\beta)=\frac{\beta}{8\pi+\beta}.    
\end{equation}
We call such steady Euler flows \emph{Onsager solutions}, as they appeared first in \cite{Onsager49}. A result of \cite{CLMP95}, rephrased with the terminology of this article, states that such solutions arise as maximizers of the same free energy $ \sfF_\beta$ in \eqref{eq:freeenergy} over the set 
\begin{equation}
\cP=\left\{\omega\in L^1: \omega\geq 0, \  \int_\DD \omega(x)\dd x =1, \ \int_{\DD} \omega(x)\log\omega(x) \dd x <\infty\right\}.
\end{equation}
Moreover, we notice by \eqref{eq:omegabeta} that as $\beta\to-8\pi$ we have $\omega_\beta\rightharpoonup\delta_0$,
see Theorem \ref{thm:caglioti} below for a precise statement. This restriction of $\beta>-8\pi$ did not apply to the previous results as we considered the vortex patch problem, which has the conservation of the $L^\infty$ norm which prevents blow-up.

Despite the importance of such solutions, their stability has not been addressed, to the best of our knowledge. We address this question in Section \ref{sec:Onsager}.
The techniques developed in the proof of Theorem \ref{thm:uniqueness} allow us to prove the following stability results for perturbations that are bounded away from 0.

\begin{theorem}\label{thm:ArnoldStability}
    For any  $\beta\in (-8\pi,\infty)$, the solution $\{\omega_\beta\}$ to \eqref{eq:MFequation}  is $L^2$ stable with respect to $L^\infty$ perturbations. 
    That is, for any $\eps>0$ and any positive $\omega^{in}\in L^\infty$,
    there exists $\delta=\delta(\eps, \|\omega^{in}\|_{L^\infty})>0$ such that if $\|\omega^{in}-\omega_\beta\|_{L^2}<\delta$, the corresponding Euler solution $\omega=\omega(t)$ is such that $\|\omega(t)-\omega_\beta\|_{L^2}<\eps$ for any $t>0$.
\end{theorem}

 The case $\beta\ge 0$ follows from the classical method of Arnold \cites{Arnold66,AK98}, as the right-hand side of \eqref{eq:MFequation} is an increasing function (see also \cite{GS24} for a recent revisitation of the method). However, the case $\beta<0$ is nontrivial, and it will be our main focus. Indeed, Arnold criteria for stability of steady Euler solutions satisfying the semilinear elliptic equation $\Delta \psi =F(\psi)$ require  that
 \begin{equation}
 -\lambda_1<F'<0\qquad \text{or}\qquad  0<F'<\infty,
 \end{equation}
 where $\lambda_1>0$ is the first eigenvalue of the Dirichlet Laplacian. Such condition is clearly violated by \eqref{eq:MFequation}.
 
 To obtain this Lyapunov stability result we make use of a quantitative Jensen's inequality and an adaptation of Talenti's original argument \cite{Talenti76}. In particular, we borrow ideas from \cite{amato2023talenti}, which combine the arguments of the quantitative versions of Polya-Szego \cite{CianchiEspositoFuscoTrombetti+2008+153+189} and Hardy-Littlewood  \cite{cianchi2008strengthened} inequalities, specialized to the solutions of the Poisson equation to obtain a quantitative version of Talenti's inequality.

We mention that the uniqueness of Onsager type solutions in the sphere $\SS^2$ was recently addressed in \cite{gui2018sphere} by studying Onofri's inequality \cite{onofri1982positivity}, which settled a conjecture in conformal geometry \cites{chang1987prescribing}. The stability of Theorem~\ref{thm:ArnoldStability} is related to the dual formulation of Onofri's inequality, which was exploited recently in \cite{carlen2023stability} to obtain stability of the log-HLS inequality.

\section{Uniqueness implies concavity}
The purpose of this section is to prove Theorem \ref{thm:main}, for vortex patches of the form \eqref{eq:vortexpatch} in a general two-dimensional simply connected domain $M\subset \RR^2$. In which case, we will make use of the characterization \eqref{eq:Ochar} for the orbit of the patch, with the disk $\mathbb{D}$ replaced by $M$. As mentioned in Section \ref{sub:mainresult}, the concavity of $\cS$  is a consequence of the min-max principle stated in Proposition \ref{prop:minmax}.
The only nontrivial requirement is stated in Proposition \ref{prop:minmax}\ref{item:minmaxunique}, which requires the uniqueness of maximizers $\omega_\beta\in\cO$ of the functional
\begin{equation}
\sfL(\omega,\beta)=\beta \sfe +\sfS_f(\omega) - \beta \sfE(\omega),
\end{equation} 
for a given fixed $\beta\in\RR$.  

In the generality of Section \ref{sub:statmech}, the energy functional $\sfE$ still achieves a maximum and minimum values $\sfe_{max}\geq\sfe_{min}\geq 0$ on $\cO$.
Theorem \ref{thm:main} is then a consequence of  Proposition \ref{prop:minmax} and the following result.

\begin{proposition}\label{prop:conditionalconcavity}
    Assume that the free energy $\sfF_\beta$ has a unique maximizer $\omega_\beta\in\cO$ for each $\beta\in\RR$. Then
     the function $\sfe\mapsto \cS(\sfe)$ is strictly concave on $[\sfe_{min},\sfe_{max}]$.
\end{proposition}

We start by deriving an Euler-Lagrange equation, that we will use in the proof of Proposition~\ref{prop:conditionalconcavity}.
\begin{lemma}\label{lem:EL1}
Assume $f\in C([0,\infty)) \cap C^1((0,\infty))$ is concave and $\lim_{z\to0^+}f'(z)=-\infty$. Any maximizer $\bar\omega$ of $\sfF_\beta$ over $\cO$ satisfies $\inf \bar\omega>0$, and there exists  $\bar\lambda=\bar\lambda(\bar\omega)$ such that
\begin{align}
\frac{1}{|M|}f'(\bar\omega)+\beta\bar{\psi}&= \bar\lambda, \qquad \text{a.e. on}\quad \{\bar\omega<1\}, \label{eq:EL'}
\end{align}
where
$$
\begin{cases}
    \Delta\bar\psi=\bar\omega & \text{in} \quad M,\\
    \bar\psi=0& \text{on} \quad \partial M.
\end{cases}
$$
\end{lemma}
\begin{proof}
To prove \eqref{eq:EL'}, we consider a positive smooth function $\varphi$, such that
\begin{equation}\label{eq:massphi}
\frac{1}{|M|}\int_M (1-\bar\omega )\varphi =1.    
\end{equation}
We take the perturbation
$$
\bar{\omega}_\eps=\frac{\sfm}{\sfm+\eps}\left(\bar\omega+\eps (1-\bar\omega)\varphi\right), 
$$
with $\eps>0$ small enough to satisfy $\bar\omega_\eps\in\cO$. Taking a variation of the entropy, and using concavity, we know
$$
\liminf_{\eps\to 0^+}\frac{\sfS_f(\bar{\omega}_\eps)-\sfS_f(\bar{\omega})}{\eps}\ge -\frac{1}{|M|}\int_{M}  f'(\bar\omega)(1-\bar\omega)\varphi  -\frac{1}{\sfm} \sfS_f(\bar{\omega}).
$$
Similarly, taking a variation of the Energy we get
\begin{equation}
\lim_{\eps\to 0^+}\frac{\sfE(\bar{\omega}_\eps)-\sfE(\bar{\omega})}{\eps}
=-\int_M \bar\psi(1-\bar\omega)\varphi  -\frac{2}{\sfm}\sfE(\bar{\omega}).
\end{equation}
Using the maximality property of $\bar\omega$, we know that
$$
\lim_{\eps\to 0^+}\frac{\sfF_\beta(\bar{\omega}_\eps)-\sfF_\beta(\bar{\omega})}{\eps}\le 0,
$$
which immediately implies that 
$$
\int_M\left(\frac{f'(\bar\omega)}{|M|}+\beta\bar\psi\right)(1-\bar\omega)\varphi \ge \bar{\lambda}(\bar\omega):= \frac{1}{\sfm} \left(\sfS(\bar{\omega})-2\beta\sfE(\bar{\omega})\right),
$$
for any positive and smooth test function $\varphi$ which satisfies \eqref{eq:massphi}. Therefore, 
\begin{equation}
\frac{f'(\bar\omega)}{|M|}+\beta\bar\psi\ge \bar{\lambda}\qquad \mbox{a.e. on $\{\bar\omega<1\}$}.
\end{equation}
Using that $f'(0)=-\infty$ and that $\bar\psi$ is uniformly bounded, we have that $\bar\omega$ is uniformly bounded below in $M$. Hence, we have that the perturbation
$$
\bar{\omega}_\eps=\frac{\sfm}{\sfm-\eps}\left(\bar\omega-\eps\varphi\right)\in \cO,
$$
for any positive smooth $\varphi$ satisfying
$$
\frac{1}{|M|}\int_M\varphi  =1
$$
and $\eps>0$ small enough. Following the same arguments as above, we obtain the reverse inequality
$$
\frac{f'(\bar\omega)}{|M|}+\beta\bar\psi\le \bar{\lambda}\qquad \mbox{a.e. on $M$},
$$
and \eqref{eq:EL'} follows.

\end{proof}

\begin{proof}[Proof of Proposition~\ref{prop:conditionalconcavity}]
Applying the min-max principle in  Proposition \ref{prop:minmax}, the maximal entropy function can be written as an infimum over affine functions in $\sfe$, namely
\begin{equation}\label{eq:minmaxSfe}
\cS(\sfe)=
\inf_{\beta\in \RR}\left\{ \beta \sfe +\max_{\omega\in \cO_{in}} \sfF_\beta(\omega) \right\}.
\end{equation}
hence it is concave in $\sfe$. We now verify its strict concavity. For any $\beta\in\RR$, let
\begin{equation}
g(\beta):=\max_{\omega\in \cO_{in}} \sfF_\beta(\omega).
\end{equation}
Being the supremum of affine functions of $\beta$, $g$ is a proper  lower semicontinuous convex function. Therefore, for any $\beta\in\RR$, the subdifferential $\de g(\beta)$ is non-empty and monotone. We define the set
\begin{equation}
\beta(\sfe):=\left\{\beta\in \RR:  \cS(\sfe)=\beta\sfe+g(\beta)\right\}.
\end{equation}
We claim that $\beta(\sfe)$ is non-empty in $(\sfe_{min},\sfe_{max})$, single-valued and monotone as a function of $\sfe$.

\medskip

\noindent $\bullet$ \textbf{$\beta(\sfe)$ is nonempty.}
For $\sfe\in(\sfe_{min},\sfe_{max})$, we first show that
\begin{equation}\label{eq:limfbeta}
\lim_{\beta\to\pm \infty}\left(\beta\sfe+g(\beta)\right)=+\infty.
\end{equation}
For $\beta>0$, we have
\begin{equation}
-\beta\sfe_{min}=\sfF_\beta(\omega_{min})\le \max_{\omega\in\cO_{in}}\sfF_\beta(\omega)=g(\beta),
\end{equation}
which implies the bound
\begin{equation}
\beta(\sfe-\sfe_{min})\le \beta\sfe+g(\beta),
\end{equation}
and shows that $\displaystyle\lim_{\beta\to\infty}\left(\beta\sfe+g(\beta)\right)=+\infty$. Similarly, for $\beta<0$, we have the bound
\begin{equation}
-\beta\sfe_{max}=\sfF_\beta(\omega_{max})\le g(\beta)),
\end{equation}
which implies the bound
\begin{equation}
-\beta(\sfe_{max}-\sfe)\le \beta\sfe+g(\beta),
\end{equation}
and the claim \eqref{eq:limfbeta} follows. If $\{\beta_n\}\subset \RR$ is a minimizing sequence such that
\begin{equation}
\cS(\sfe)=
\lim_{n\to\infty}\left( \beta_n \sfe +g(\beta_n) \right),
\end{equation}
then by  \eqref{eq:limfbeta} we must have that $\{\beta_n\}$ is bounded and therefore has a limit point $\bar\beta$. By the lower semicontinuity of $f$ and , we find
from \eqref{eq:minmaxSfe} that $ \bar\beta \sfe +g(\bar\beta)= \cS(\sfe)$, proving the claim.

\medskip

\noindent $\bullet$ \textbf{Characterization.}
 If $\beta\in\beta(\sfe)$, then 
\begin{equation}\label{eq:charsubdiff0}
-\sfe\in\partial g(\beta),
\end{equation}
because $\beta$ is a minimizer. Next we show that the subdifferential of $f$ is given by
\begin{equation}\label{eq:charsubdiff}
\partial g(\beta)=\left\{-\sfE(\omega_\beta)\right\},
\end{equation}
where $\omega_\beta$ is the unique maximizer of $\sfF_\beta$ over $\cO_{in}$. Since $g$ is the pointwise supremum of affine functions $\sfG_\omega(\beta):=\sfS_f(\omega)-\beta\sfE(\omega)$ over $\omega\in\cO_{in}$,
the subdifferential $\de g(\beta)$ is in general \cite{Zalinescu02}*{Theorem 2.4.18} given by
\begin{equation}
\partial g(\beta)={\rm co}\left(\bigcup_{\omega\in \cO_{in}}\left\{\de \sfG_\omega(\beta): \sfG_\omega(\beta)=g(\beta)\right\}\right),
\end{equation}
where ${\rm co}(B)$ denotes the closed convex hull of the set $B$. Due to our uniqueness assumption and the fact that $\sfG_\omega$ is differentiable, in our case the above identity reduces to
\begin{equation}
\partial g(\beta)=\left\{\frac{\dd}{\dd\beta}G_{\omega}(\beta)\Big|_{\omega=\omega_\beta}\right\} =\left\{-\sfE(\omega_\beta)\right\},
\end{equation}
which is  \eqref{eq:charsubdiff}.

\medskip

\noindent $\bullet$ \textbf{Strict monotonicity.}
Thanks to \eqref{eq:charsubdiff0}-\eqref{eq:charsubdiff},  the function $\beta(\sfe)$ is given implicitly by the equation
\begin{equation}
\sfE(\omega_{\beta(\sfe)})=\sfe.
\end{equation}
Hence, $\beta(\sfe)$ is well-defined and monotone if we can prove that the mapping
$\beta\to \sfE(\omega_{\beta})$ is strictly monotone.

Given $\beta_1<\beta_2$, we want to show that $\sfE(\omega_{\beta_1})>\sfE(\omega_{\beta_2})$. By the Euler-Lagrange condition \eqref{eq:EL'}, we know that  $\omega_{\beta_1}\ne \omega_{\beta_2}$. Hence, by uniqueness of maximizers, we have
\begin{equation}\label{eq:ineq}
\sfF_{\beta_1}(\omega_{\beta_1})>\sfF_{\beta_1}(\omega_{\beta_2}),\qquad
\sfF_{\beta_2}(\omega_{\beta_2})>\sfF_{\beta_2}(\omega_{\beta_1})
\end{equation}
which implies
\begin{equation}
(\beta_2-\beta_1)(\sfE(\omega_{\beta_2})-\sfE(\omega_{\beta_1}))<0.
\end{equation}
This shows the desired strict monotonicity.
\end{proof}

\begin{remark}
The assumption that $\omega^{in}=\ind_A$ for some $A\subset M$ is only used to show that for any $\beta_1\neq\beta_2$ we have that $\omega_{\beta_1}\neq\omega_{\beta_2}$, which stems from the Euler-Lagrange condition \eqref{eq:EL'}.
\end{remark}

\section{Free energy and entropy maximizers}
In this section we analyze various aspects of entropy maximization related to the free energy $\sfF_\beta$. 
We specialize from now on in setting of Section \ref{sub:mainresult}, hence considering the vortex patch problem in the unit disk $M=\DD$, with entropy given by 
\eqref{eq:boltzDisc}.
Since our analysis is based on
rearrangements inequality, we first take some time to review the basic concepts of radial rearrangements. The assumptions on $f$ are those of Theorem \ref{thm:main}

\subsection{Rearrangements and radial symmetry of maximizers}
A standard technique for studying maximizers is utilizing rearrangements of mass. 
Given a set $B\subset \DD$, its symmetric rearrangement $B^\sharp$ is the open
centered ball whose volume agrees with $B$, namely
\begin{equation}
B^\sharp=\left\{x\in\RR^2: \pi|x|^2<|B|\right\}.
\end{equation}
Given a function $\omega\in \cO$, its \emph{symmetric decreasing rearrangement}
is defined by
\begin{equation}
\omega^\sharp(x)=\int_0^1 \ind_{[\omega>t]^\sharp}(x)\dd t.
\end{equation}
Notice that $\omega^\sharp$ is radial. The \emph{symmetric increasing rearrangement} of $\omega$ is 
\begin{equation}
\omega_\sharp(x)=\omega^\sharp\left(\sqrt{1-|x|^2}\right).
\end{equation}
It follows directly from a theorem of Talenti \cite{Talenti76}*{Theorem 1(v)} that 
\begin{equation}\label{eq:energytalent}
\sfE(\omega)=\frac12\| \nabla\Delta^{-1}\omega\|^2_{L^2} \leq \frac12\| \nabla \Delta^{-1}\omega^\sharp\|^2_{L^2}   =\sfE(\omega^\sharp),
\end{equation}
for any $\omega\in \cO$,  with equality if and only if $\omega=\omega^\sharp$, see \cites{CL92, Kesavan06}.  We are also interested in comparing the kinetic energy among \emph{radial} vorticities.
If $\omega$ is radial, then the corresponding streamfunction can be explicitly derived from \eqref{eq:stream} as
\begin{equation}\label{eq:psiradial}
\psi(r)=-\int_{r}^1\frac{1}{s}\int_0^s\omega(\bar{s})  \bar{s}\,\dd\bar{s}\,\dd s
\end{equation}
For any radial function $g\in \cO$, define
\begin{equation}\label{eq:massrad}
M_g(r)=2\pi \int_0^r g(s)s\,\dd s = \int_{B_r} g(x)\dd x, \qquad r\in (0,1],
\end{equation}
where $B_r$ is the ball of radius $r$ centered at the origin. Thanks to \eqref{eq:psiradial}, for any radial $\omega\in \cO$, it is not hard to see that
\begin{equation}\label{eq:energyM}
\sfE(\omega)
=\pi \int_0^1 |\de_r \psi(r)|^2 r\dd r
=\frac{1}{4\pi} \int_0^1 \frac1r \left| M_\omega(r)\right|^2 \dd r.
\end{equation}
Moreover, we have the following comparison principle in the radial case.
\begin{lemma}
Let $\omega\in \cO$ be a radial function. Then
\begin{equation}
\sfE(\omega_\sharp)\leq\sfE(\omega)\leq\sfE(\omega^\sharp).
\end{equation}
\end{lemma}

\begin{proof}
Thanks to the Hardy-Littlewood inequality\footnote{For any two measurable functions $f,g:\DD\to [0,\infty)$, there 
hold $$\int_{\DD}f^\sharp g_\sharp \leq \int_{\DD} fg \leq \int_{\DD}f^\sharp g^\sharp$$} and the fact that $(\ind_{B_r})^\sharp = \ind_{B_r}$,
we have
\begin{equation}
\int_{\DD} \ind_{B_r}(x) \omega_\sharp(x)\dd x\leq \int_{\DD} \ind_{B_r}(x) \omega(x)\dd x\leq \int_{\DD} \ind_{B_r}(x) \omega^\sharp(x)\dd x,
\end{equation}
implying that $M_{\omega_\sharp}\leq M_\omega\leq M_{\omega^\sharp}$. The claim follows from \eqref{eq:energyM}.
\end{proof}

The representation \eqref{eq:energyM} of $\sfE$ is also useful to compute explicitly the energy for specific vorticities. For instance,
\begin{equation}\label{eq:constantvort}
\omega_\star\equiv \sfm \qquad \Rightarrow \qquad \sfe_\star:=\sfE(\omega_\star) = \frac{\pi\sfm^2}{16}.
\end{equation}
Now, for any radial $\omega\in\cO$, we have that $M_\omega(r)\leq\pi\min\{r^2,\sfm\}$. Defining
\begin{equation}
\omega_{max}(r)=\begin{cases}
1,\quad r\in\left(0,\sqrt{\sfm}\right),\\
0,\quad r\in\left(\sqrt{\sfm},1\right),
\end{cases}
\end{equation}  
we then have that $M_\omega\leq M_{\omega_{max}}$. Similarly, 
\begin{equation}
M_\omega(r)=M_\omega(1)-2\pi \int_r^1 \omega(s)s\,\dd s\geq  \pi \left(\sfm-1+r^2\right),
\end{equation}  
so that $M_\omega\geq M_{\omega_{min}}$, where
\begin{equation}\label{eq:omegamin}
	\omega_{min}(r)=\begin{cases}
	0,\quad r\in\left(0,\sqrt{1-\sfm}\right),\\
	1,\quad r\in\left(\sqrt{1-\sfm},1\right).
	\end{cases}
\end{equation}  
As a consequence, by a direct computation of $\sfE(\omega_{min})$ and $\sfE(\omega_{max})$ we have the following result, which we state
without proof.

\begin{lemma}\label{lem:energymaxmin}
For any radial $\omega\in \cO$ we have
\begin{equation}
\sfe_{min}\leq\sfE(\omega)\leq\sfe_{max},
\end{equation}
where
        \begin{equation}
        \sfe_{min}:=\sfE(\omega_{min})=\frac{\pi\sfm^2}{16} -\frac{\pi}{8}(1-\sfm)\left(\sfm+(1-\sfm)\log(1-\sfm)\right)
        \end{equation}
and
        \begin{equation}
        \sfe_{max}:=\sfE(\omega_{max})=\frac{\pi\sfm^2}{16} +\frac{\pi \sfm^2}{8}|\log\sfm|.
        \end{equation}
\end{lemma}

In fact, the above functions satisfy the stronger property below.
\begin{lemma}\label{lem:singletons}
The functions $\omega_{min}$ and $\omega_{max}$ are the unique functions that achieve their energy levels, that is
\begin{align}
\label{eq:singletons1}\sfE(\omega)=\sfe_{min} \quad &\Rightarrow \quad \omega=\omega_{min},\\
\label{eq:singletons2}\sfE(\omega)=\sfe_{max} \quad &\Rightarrow \quad \omega=\omega_{max}.
\end{align}
\end{lemma}
\begin{proof}
Since $\omega\mapsto\sfE(\omega)$ is a convex function, it has a unique global minimizer in $\cO$. Moreover, such minimizer is radially symmetric,
since $\sfE$ is invariant under rotation. Thus, Lemma \ref{lem:energymaxmin} implies that $\omega_{min}$ is the unique minimizer. Turning to \eqref{eq:singletons2},
we know by \eqref{eq:energytalent} that any energy maximizer is necessarily radially decreasing, and by Lemma \ref{lem:energymaxmin} that $\omega_{max}$ is one of them.
If $\bar\omega$ is another radially decreasing maximizer, it is not hard to see that $\bar\omega=\omega_{max}$ if and only if $M_{\bar\omega}=M_{\omega_{max}}$, i.e. if and only if 
$\sfE(\bar\omega)= \sfE(\omega_{max})$.
\end{proof}

\subsection{Relaxed maximization problems}
The strategy to prove Theorem \ref{thm:concavity} is to study (relaxed versions of) the maximization problem \eqref{pb:maxpb} and  
apply a min-max principle. In turn, we will see how this is reduced to study uniqueness of maximizers for the free energy $\sfF_\beta$ in
\eqref{eq:freeenergy}, as $\beta\in \RR$ varies.
We begin with the following observations.

\begin{lemma}\label{lem:existmax}
For every $\sfe\in[\sfe_{min},\sfe_{max}]$, the constrained maximization problem \eqref{pb:maxpb} admits at least one solution.
\end{lemma}
\begin{proof}
Notice that  $\sfS_f$ is strictly concave and upper semicontinuous. Let $\omega_n$ be a maximizing sequence. Since $\cO$ is weak-$*$ compact, there exists $\bar\omega\in\cO$
such that $\omega_n\xrightharpoonup{*} \bar\omega$ up to subsequences,
and since $\sfE(\omega_n)\to \sfE(\bar\omega)$ by continuity, then $\sfE(\bar\omega)=\sfe$. Moreover, 
\begin{equation}
\sup \left\{\sfS_f(\omega): \omega\in \cO,\ \sfE(\omega) =\sfe  \right\}=\limsup_{n\to\infty}\sfS_f(\omega_n)\leq \sfS_f(\bar\omega),
\end{equation}
as $\sfS_f$ is upper semicontinuous.
Hence, $\bar\omega$ is a maximizer.  
\end{proof}

It is also important to identify the global entropy maximizer, regardless of the energy constraint. 

\begin{lemma}\label{lem:unconmax}
The unconstrained maximization problem 
\begin{equation}
\text{maximize } \sfS_f(\omega)\text{ subject to } \omega\in \cO,
\end{equation}  
has a unique maximizer given by the constant solution $\omega_\star\equiv \sfm$.
\end{lemma}
\begin{proof}
The proof follows from the previous Lemma \ref{lem:existmax}, with the uniqueness stemming from the strict concavity of the entropy functional.
Assume now that $\bar\omega$ is the unique maximizer, and let $\Phi:\DD\to\DD$ be a volume-preserving diffeomorphism. Then $\bar\omega\circ \Phi \in \cO$,
and $\sfS_f(\bar\omega\circ \Phi)=\sfS_f(\bar\omega)$. Thus, by uniqueness, $\bar\omega\circ \Phi =\bar\omega$, and since $\Phi$ is arbitrary, this implies $\bar\omega$ is constant,
hence equal to its average.
\end{proof}
Thanks to the computation in \eqref{eq:constantvort}, Theorem \ref{thm:concavity}\ref{item:maximum} is proved. The role of $\sfe_\star=\sfE(\omega_\star)$ in
 \eqref{eq:constantvort} is quite interesting. Indeed, for a fixed energy level $\sfe\leq \sfe_\star$, the maximization problem \eqref{pb:maxpb} can be relaxed into
 a convex one.
 
 \begin{lemma}\label{lem:1relax}
    If   $\sfe_{min}\le \sfe\le \sfe_\star$, then the constrained maximization problem \eqref{pb:maxpb} is equivalent to the relaxed problem
\begin{equation}\label{pb:maxpbrelLEQ}
\text{maximize } \sfS_f(\omega)\text{ subject to } \omega\in \cO,\ \sfE(\omega) \leq\sfe.
\end{equation} 
In particular, problem \eqref{pb:maxpbrelLEQ}, and hence \eqref{pb:maxpb}, have a unique solution.

If $ \sfe_\star\le \sfe\le \sfe_{max}$,
then the constrained maximization problem \eqref{pb:maxpb} is equivalent to the relaxed problem
\begin{equation}\label{pb:maxpbrelGEQ}
\text{maximize } \sfS_f(\omega)\text{ subject to } \omega\in \cO,\ \sfE(\omega) \geq\sfe.
\end{equation} 
\end{lemma}

\begin{proof}
We first look at the case $\sfe\leq\sfe_\star$ and argue by contradiction. Fix any maximizer $\bar\omega\in\cO$ with $\sfE(\bar\omega)< \sfe$. For any $t\in[0,1]$, the convex combination 
 $\omega_t=(1-t)\bar\omega+t\omega_\star$ is in $\cO$. Moreover, using the continuity of $\sfE$, we have that for $t$ small enough
 $\sfE(\omega_t)< \sfe$.
Since $\bar\omega$ is a maximizer,
\begin{equation}
    \left.\frac{\dd}{\dd t}\sfS_f(\omega_t)\right|_{t=0}\leq 0.
\end{equation}
However, a direct computation shows
\begin{equation}
    \left.\frac{\dd}{\dd t}\sfS_f(\omega_t)\right|_{t=0}=\frac{1}{|\DD|}\int_\DD (\omega_\star-\bar\omega)f'(\bar\omega) 
    =\frac{1}{|\DD|}\int_\DD (\omega_\star-\bar\omega)(f'(\bar\omega) -f'(\omega_\star))>0,
\end{equation}
hence reaching a contradiction. Since the condition $ \sfE(\omega) \leq\sfe$ in \eqref{pb:maxpbrelLEQ} is convex and the entropy functional is strictly concave, uniqueness of solution follows immediately.
The case $\sfe\geq\sfe_\star$ is completely analogous, however \eqref{pb:maxpbrelGEQ} is not a convex problem and therefore uniqueness of solutions is not as immediate
as the previous case.
\end{proof}

\section{Uniqueness of maximizers at negative temperature}
In light of Proposition \ref{prop:conditionalconcavity}, we have reduced the problem to the investigation of uniqueness of maximizers for the free energy $\sfF_\beta$.
This is the main result of this section.
\begin{theorem}\label{thm:beta<0}
For any $\beta\in\RR$, there exists a unique maximizer $\omega_\beta$ of $\sfF_\beta$ over $\cO$. 
\end{theorem}
For $\beta\geq 0$,  $\sfF_\beta$ is strictly concave and therefore admits a unique maximizer, which is necessarily radial due to the invariance of the functionals under rotations. Hence, this section is devoted to the study of the case of negative inverse temperature $\beta<0$.

\subsection{The Euler-Lagrange equations}

Despite the loss of concavity of $\sfF_\beta$ for $\beta<0$, 
Talenti's inequality \eqref{eq:energytalent} implies that if $\omega$ is not radially symmetric and decreasing then
\begin{equation}\label{eq:rearr}
\sfF_\beta(\omega)=\sfS(\omega)-\beta \sfE(\omega)=\sfS(\omega^\sharp)-\beta \sfE(\omega)<\sfS(\omega^\sharp)-\beta \sfE(\omega^\sharp)=\sfF_\beta(\omega^\sharp).
\end{equation}
Thus, any maximizer is radially symmetric and decreasing. First, we derive the Euler-Lagrange equations, which we refine from Lemma~\ref{lem:EL1}.

\begin{lemma}\label{lem:EL}
Let $\beta<0$. Any maximizer $\bar\omega$ of $\sfF_\beta$ over $\cO$ is radially decreasing and satisfies the following:
There exist  $r_{\bar\omega}\in[0,1)$ and $\bar\lambda=\bar\lambda(\bar\omega)\in[0,\infty)$ such that
\begin{align}
\bar\omega(r)&=1, \qquad \text{on}\quad [0,r_{\bar\omega}),\label{eq:EL1}\\
-\log\bar\omega(r) -\frac{\beta}{2} \int_r^1\frac{1}{s}M_{\bar\omega}(s) \dd s &=\bar\lambda, \qquad \text{on}\quad (r_{\bar\omega},1). \label{eq:EL2}
\end{align}
As a consequence, $\bar\omega(r)\ge \e^{-\bar{\lambda}} >0$ on $[0,1]$. 
\end{lemma}

\begin{proof}
By \eqref{eq:rearr} we know that any maximizer must be radially symmetric and decreasing. Hence, the first property that there 
exists a $r_{\bar\omega}\in[0,1)$, such that the set $\{\bar\omega=1\}=B_{r_{\bar\omega}}$ follows immediately.

To prove \eqref{eq:EL2}, we simply rely on Lemma \ref{lem:EL1}. In the case of Boltzmann entropy, $f'(z)=-1-\log(z)$. Moreover, we can make use of the expression \eqref{eq:psiradial} for radial streamfunctions with \eqref{eq:massrad} and conclude the proof.
\end{proof}

We use the non-degeneracy that stems from \eqref{eq:EL2}, to show that out of the set of possible maximizers $\bar\omega$, there exists at least one maximizer with the smallest possible radius $r_{\bar\omega}$.

\begin{lemma}\label{lem:minradius}
Let $\beta<0$. There exist a maximizer $\omega_\beta$ of  $\sfF_\beta$ over $\cO$ such that
\begin{enumerate}[label=(\alph*), ref=(\alph*)]
\item $\{\omega_\beta=1\}=[0,r_\beta)$; \label{item a}
\item For $\lambda_\beta=\bar{\lambda}(\omega_\beta)$ in \eqref{eq:EL2}, there holds
\begin{equation}\label{eq:ELomegabeta}
-\log\omega_\beta (r) -\frac\beta2 \int_r^1\frac{1}{s}M_{\omega_\beta}(s) \dd s =\lambda_\beta, \qquad \text{on } (r_{\beta
},1);
\end{equation}

\item $r_\beta\leq r_{\bar\omega}$ for any maximizer $\bar\omega$. \label{item c}
\end{enumerate}
\end{lemma}

\begin{proof}
By the upper-semicontinuity of the entropy and the continuity of the energy, it follows that the set of maximizers
$$
\mathcal{M}=\left\{\bar\omega\in\cO\;:\; \sfF_\beta(\bar\omega)=\max_{\omega\in \cO}\sfF_\beta(\omega)\right\}
$$
is weak-$*$ compact. As \ref{item c} suggests, we want $r_\beta$ to be defined by
$$
r_\beta=\inf_{\bar\omega\in \mathcal{M} } r_{\bar\omega}.
$$
We consider $\{\bar\omega_n\}_{n\in\NN}\subset \mathcal{M}$ an minimizing sequence such that
$$
\lim_{n\to\infty} r_{\bar\omega_n}=r_\beta.
$$
Using compactness, we know that, up to subsequence, there exists an accumulation point $\bar\omega_\beta\in\mathcal{M}$ such that $\bar\omega_n\xrightharpoonup{*} \omega_\beta$. The fact that $\omega_\beta$ satisfies the Euler-Lagrange condition \eqref{eq:ELomegabeta} follows from the maximality, so that to complete the proof we just need to show \ref{item a}, namely
$$
\{\omega_\beta=1\}=[0,r_\beta).
$$
We notice that $\{\bar\omega_n\}_{n\in\NN}$ is a sequence of radially decreasing and bounded functions, hence they are  
of bounded variation, which implies the convergences is strong in $L^p(\DD)$ for any $p\in[1,\infty)$. Using that $r_\beta\le r_{\bar\omega_n}$, 
we know that $\bar\omega_n=1$ on $r\in[0,r_\beta)$, which implies the inclusion $[0,r_\beta)\subset \{\omega_\beta=1\}$, for instance by passing to a further subsequence that converges pointwise. We will show the 
reverse inclusion by using the Euler-Lagrange equation for $\bar\omega_n$. Indeed, differentiating in $r$ the Euler-Lagrange condition \eqref{eq:EL2} for $\bar\omega_n$, we get the equation
$$
-\frac{\de_r\bar\omega_n(r)}{\bar\omega_n(r)}+\frac\beta2 \frac{1}{r}M_{\omega_{\beta}}(r)=0,\qquad\text{for } r>r_n.
$$
Using the strong convergence of $\bar\omega_n\to\omega_\beta$,  and the distributional convergence of the derivatives $\de_r\bar\omega_n\rightharpoonup\de_r\omega_\beta$, we obtain that 
$$
\de_r\omega_\beta(r)=\frac\beta2\omega_\beta(r) \frac{1}{r}M_{\omega_{\beta}}(r)<0, \qquad\text{for } r>r_\beta,
$$
which implies that 
$$
\omega_\beta<1  \qquad\mbox{for $r>r_\beta$},
$$ 
and the proof is finished.
\end{proof}
Written for the corresponding radial streamfunction $\psi_\beta$, the Euler-Lagrange equation \eqref{eq:ELomegabeta} reads
\begin{equation}\label{eq:ELpsibeta}
\Delta\psi_\beta = \e^{-\lambda_\beta} \e^{\pi\beta \psi_\beta}.
\end{equation}

\subsection{Uniqueness of maximizers}
We will show that $\omega_\beta$ is in fact unique maximizer of the free energy $\sfF_\beta$. We start by considering $\bar{\omega}$ a general radial competitor. 
Below we write an expression for $\sfF_\beta(\bar{\omega})$ in terms of $\omega_\beta$ and the Brenier \cite{Brenier91} (or optimal transport map) 
between them. As both functions are radial (and decreasing), we know that the optimal mapping is also radial and increasing. Just like the one dimensional case it can be 
represented implicitly by the cumulative distribution functions. To avoid some  pathological regularity situations, from now on we assume that both the source and target measure are bounded above and below, which is satisfied by maximizers, see Lemma~\ref{eq:ELomegabeta}. Namely, there exists a unique  strictly increasing map $T:[0,1]\to [0,1]$ such that $T(0)=0$, $T(1)=1$ and
\begin{equation}\label{eq:breniemap}
\int_0^{r}\omega_\beta(s)s\,\dd s=\int_0^{T(r)}\bar\omega(s)s\,\dd s,\qquad \text{for any } r\in[0,1].
\end{equation}
Using the Monge-Ampere equation associated to the change of variable we obtain the relationship
\begin{equation}\label{eq:MA}
\frac{r}{T(r)T'(r)}\omega_\beta(r)=\bar\omega(T(r)),\qquad \text{for any } r\in[0,1].
\end{equation}
To simplify the notation we define the function
\begin{equation}\label{eq:phi}
\phi(r):=\frac{T(r)T'(r)}{r}.
\end{equation}
The next results is a comparison between the energy and entropy of $\omega_\beta$ and $\bar\omega$.
\begin{lemma}\label{lem:Breniermap}
Let $\omega_\beta$ and $\bar{\omega}$ be two radial functions in $\cO$ that are bounded below away from zero. Then
\begin{equation}\label{eq:BrenierH}
\sfS(\bar{\omega})-2\int_0^1  \omega_\beta(r) \log \phi(r)  r\dd r=\sfS(\omega_\beta)
\end{equation}
\begin{equation}\label{eq:BrenierE}
\sfE(\bar{\omega})+\int_0^1  \left(\int_r^1\frac{1}{s} \omega_\beta(s) M_{\omega_\beta}(s)\dd s
\right)\log \phi(r)r\dd r\le \sfE(\omega_\beta),
\end{equation}
where $\phi(r)$ is defined in \eqref{eq:phi}.
\end{lemma}
\begin{remark}
We note that in the previous result we do not use the optimality in of $\bar{\omega}$ or $\omega_\beta$ in any strong way, it only requires that both vorticities are radial.
\end{remark}
\begin{proof}
Using the change of variable $r=T(s)$ given by the Brenier map \eqref{eq:breniemap}, we can rewrite the entropy as follows
$$
\sfS(\bar{\omega})=-2\int_0^1 \bar{\omega}(r)\log(\bar{\omega}(r))r\dd r=-2\int_0^1 \bar{\omega}(T(s))\log(\bar{\omega}(T(s)))T(s) T'(s)\dd s.
$$
Using \eqref{eq:MA}, we obtain
$$
\sfS(\bar{\omega})=-2\int_0^1 \omega_\beta(s)\log\left(\frac{\omega_\beta(s)}{\phi(s)}\right)s\,\dd s,
$$
which coincides with the desired \eqref{eq:BrenierH}.

For the energy, we first integrate by parts to obtain the different representation 
$$
\sfE(\bar{\omega})=\frac{1}{4\pi}\int_0^1 \frac{|M_{\bar{\omega}}(r)|^2}{r}\dd r= -\int_0^1M_{\bar{\omega}}(r) \bar{\omega}(r) \log(r) r \dd r.
$$
Next, we perform the change of variables $r=T(s)$ and use \eqref{eq:MA} to obtain
\begin{equation}\label{eq:EwithT}
\sfE(\bar{\omega})=-\int_0^1 M_{\bar{\omega}}(T(s)) \bar{\omega}(T(s)) \log(T(s))T(s)T'(s) \dd s=-\int_0^1M_{\omega_\beta}(s) \omega_\beta(s)  \log(T(s)) s \dd s.
\end{equation}
We use that $\phi(r)=[T^2]'/2r$, to rewrite
$$
\log(T(s))=\frac{1}{2}\log(T^2(s))=\frac{1}{2}\log\left(2\int_0^s  \phi(a) a\dd a \right).
$$
Then, we normalize the integral to be able to apply Jensen's inequality and we obtain
$$
\log(T(s))=\frac{1}{2}\log(s^2) +\frac{1}{2}\log\left(\frac{2}{s^2}\int_0^s  \phi(a) a\dd a \right)\ge \log(s)+\frac{1}{s^2}\int_0^s \log(\phi(a)) a\dd a.
$$
Using this inequality on \eqref{eq:EwithT}, we have 
$$
\sfE(\bar{\omega})\le \sfE(\omega_\beta)-\int_0^1 \left(\frac{1}{s^2} \int_0^s  \log(\phi(a))a\dd a \right)M_{\omega_\beta}(s)\omega_\beta(s)s \;\dd s.
$$
The result \eqref{eq:BrenierE} follows by applying Fubini's theorem.
\end{proof}

Finally, we show that $\omega_\beta$ is the unique maximizer of $\sfF_\beta$.
\begin{lemma}\label{lem:last}
Let $\beta<0$. Let $r_\beta$ and $\omega_\beta$ be given by Lemma~\ref{lem:minradius}. Assume that $\bar\omega\in\cO$ is radial, 
bounded below away from 0, and satisfies that $(0,r_\beta)\subset \{\bar{\omega}=1\}$. Then we have the inequality
\begin{equation}\label{eq:energydiff}
\sfF_\beta(\bar{\omega})-2 \omega_\beta(1) \int_{0}^1\log\phi(r)r\dd r\le \sfF_\beta(\omega_\beta),
\end{equation}
where $\phi(r)$ is defined in \eqref{eq:phi}.
Moreover, if $\bar{\omega}\ne \omega_\beta$, then the energy difference is positive, namely
\begin{equation}\label{eq:posdef}
-2\omega_\beta(1)\int_{0}^1\log\phi(r)r\dd r> 0.
\end{equation}
\end{lemma}
\begin{proof}
Applying Lemma~\ref{lem:Breniermap}, we have
\begin{equation}\label{eq:aux1}
\sfF_\beta(\bar{\omega})+\int_0^1\left(-2 \omega_\beta(r) -\beta\int_r^1 \frac{ 1}{s}\omega_\beta(s) M_{\omega_\beta}(s)\dd s\right)\log(\phi(r))r\dd r\le \sfF_\beta(\omega_\beta) .
\end{equation}
Notice that by the hypothesis $(0,r_\beta)\subset \{\bar{\omega}=1\}$, Brenier's map is trivial on $[0,r_\beta)$. That is to say
\begin{equation}
\phi(r)=1, \qquad \text{on } [0,r_\beta).
\end{equation}
For   $r\in(r_\beta,1)$, we can use the Euler-Lagrange equation \eqref{eq:ELomegabeta} to simplify the remainder. More specifically, taking a derivative of \eqref{eq:ELomegabeta} to obtain
$$
\de_r\omega_\beta(r)=\frac{\beta}{2}\frac{ \omega_\beta(r) M_{\omega_\beta}(r)}{r}.
$$
Integrating back on $(r,1)$, we deduce that
\begin{equation}\label{eq:appliedEL}
\omega_\beta(1)=\omega_\beta(r)+\frac\beta2\int_r^1\frac{ 1}{s}\omega_\beta(s) M_{\omega_\beta}(s)\dd s, \qquad\text{for any } r\in(r_\beta,1).
\end{equation}
Replacing back \eqref{eq:appliedEL} into \eqref{eq:aux1}, we obtain the desired \eqref{eq:energydiff}.

To show \eqref{eq:posdef}, we apply Jensen's inequality
\begin{equation}\label{eq:Jensen}
    \frac{1}{2}\int_{0}^1\log\phi(r) r\dd r\le \log\left(\frac{1}{2}\int_{0}^1\phi(r)r\dd r\right)=\log\left(\frac{1}{2}\int_{0}^1T(r) T'(r)\dd r\right)=\log(T^2(1)-T^2(0))=0.
\end{equation}
The equality in Jensen's inequality can only occur if $\phi(r)=C$ is constant, which implies that Brenier's map $T(r)=r$ 
is the identity. The conclusion that the defect is positive if $\bar{\omega}\ne\omega_\beta$ follows directly from the previous argument, and the 
fact that $\omega_\beta(1)>0$ by Lemma~\ref{lem:EL}.
\end{proof}

We conclude this section with the proof of Theorem~\ref{thm:beta<0}, which is a consequence of the results above.

\begin{proof}[Proof of Theorem~\ref{thm:beta<0}]
We will show that $\omega_\beta$ given in Lemma~\ref{lem:minradius} is the unique maximizer of $\sfF_\beta$. Assume $\bar\omega$ is also maximizer of $\sfF_\beta$, then by Lemma~\ref{lem:EL} and Lemma~\ref{lem:minradius} it satisfies the hypothesis of Lemma~\ref{lem:last}. Applying Lemma~\ref{lem:last}, and $\sfF_\beta(\bar\omega)=\sfF_\beta(\omega_\beta)$ we obtain $\omega_\beta=\bar{\omega}$.
\end{proof}

\section{Non-radial energy maximizers at fixed angular momentum}\label{sec:angmom}

The section is dedicated to the proof of Theorem \ref{thm:nonrad}. Section~\ref{sec:upperbound} contains the upper bound on the kinetic energy $\sfE$ for radial functions on $\cO_\LL$, and Section~\ref{sec:lowerbound} lower bound by computing the energy of an explicit vortex patch. The conclusion follows then by choosing $\LL$ in terms of the angular momentum. As in the statement of Theorem \ref{thm:nonrad}, we let $\sfm\in (0,1)$ and $\LL\geq 1$ and set
\begin{equation}\label{eq:OcharL2}
\cO_\LL:=\left\{\omega\in L^\infty: 0\leq \omega\leq \LL, \ \frac{1}{|\DD|}\int_\DD \omega(x)\dd x =\sfm\right\}.
\end{equation}
The proof is carried out in the next sections.
\subsection{Upper bounds on the kinetic energy for radial functions}\label{sec:upperbound}
We claim that there exists a constant $C\geq 1$, independent of $\LL,\sfa$, and $\sfm$ such that
\begin{equation}\label{eq:enerupbound}
    \sup \left\{\sfE(\omega):\omega\in\cO_\LL,\; \sfA(\omega)=\sfa,\ \omega\ \mbox{is radial}\right\}
    \le C\left(\sfm |\sfa|+|\sfa|^2\log(\LL/|\sfa|)\right).
\end{equation}
Now, if $\omega\in \cO_\LL$ is a radial function, from \eqref{eq:angmom} we deduce that
\begin{equation}
\sfA(\omega)
=-\frac12\left[2\pi\int_0^1 \omega(r) (1-r^2)r\dd r\right]
=-\frac12\left[\pi \sfm - \int_0^1 \de_r M_\omega(r) r^2\dd r\right]
=-\int_0^1 M_\omega(r) r\dd r.
\end{equation}
Thus, for every $r\in [0,1]$, we have the pointwise identity
\begin{equation}
\frac12 M_\omega(r)(1-r^2)= \frac12\int_0^r \de_s \left[M_\omega(s)(1-s^2)\right]\dd s
=\pi\int_0^r \omega(s)(1-s^2)s\dd s-\int_0^r M_\omega(s) s\dd s.
\end{equation}
In particular,
\begin{equation}
\frac12 M_\omega(r)(1-r^2)\leq \pi\int_0^r \omega(s)(1-s^2)s\dd s\leq  -\sfA(\omega),
\end{equation}
so that any radial function $\omega\in \cO_\LL$ with $\sfA(\omega)=\sfa$ satisfies
\begin{equation}
M_\omega(r)\le \min\left\{ \pi  \LL r^2,\frac{2|\sfa|}{1-r^2},\pi \sfm\right\}.
\end{equation}
This implies
\begin{equation}
M_\omega(r)\le \begin{cases}
     \pi  \LL r^2, & r^2\le 1-\sqrt{1-\frac{8|\sfa|}{\pi\LL}},\\
    \frac{2|\sfa|}{1-r^2},& 1-\sqrt{1-\frac{8|\sfa|}{\pi\LL}}<r^2\le1-\frac{2|\sfa|}{\pi\sfm},\\
    \pi\sfm, \quad& 1-\frac{2|\sfa|}{\pi\sfm}<r^2\le 1.
\end{cases}
\end{equation}
Thanks to \eqref{eq:energyM}, we then have $\sfE(\omega)\leq \sfE_1+\sfE_2+\sfE_3$, where
\begin{align}
\sfE_1
=\frac{1}{4\pi} \int_0^{r^2=1-\sqrt{1-\frac{8|\sfa|}{\pi\LL}}} \frac1r \left|\pi  \LL r^2\right|^2 \dd r 
=\frac{\pi  \LL^2}{16} \left(1-\sqrt{1-\frac{8|\sfa|}{\pi\LL}}\right)^2  \lesssim |\sfa|^2,
\end{align}
and
\begin{align}
\sfE_2
&=\frac{1}{4\pi} \int_{r^2=1-\sqrt{1-\frac{8|\sfa|}{\pi\LL}}}^{r^2=1-\frac{2|\sfa|}{\pi\sfm}} \frac1r \left| \frac{2|\sfa|}{1-r^2}\right|^2 \dd r\\
&=\frac{|\sfa|^2}{2\pi}\left[\frac{\pi\sfm}{2|\sfa|} +  \log\left( \frac{\pi\sfm}{2|\sfa|} -1\right) -\frac{1}{\sqrt{1-\frac{8|\sfa|}{\pi\LL}}} - \log\left(\frac{1-\sqrt{1-\frac{8|\sfa|}{\pi\LL}}}{\sqrt{1-\frac{8|\sfa|}{\pi\LL}}}\right)
\right]\\
&\lesssim \sfm |\sfa|+|\sfa|^2\log(\LL/|\sfa|),
\end{align}
and
\begin{align}
\sfE_3=\frac{1}{4\pi} \int_{r^2=1-\frac{2|\sfa|}{\pi\sfm}}^1 \frac1r \left| \sfm\right|^2 \dd r
=-\frac{ \pi\sfm^2 }{8}   \log\left(1-\frac{2|\sfa|}{\pi\sfm}\right)\lesssim \sfm |\sfa|.
\end{align}
Thus \eqref{eq:enerupbound} follows by collecting the above three bounds.

\subsection{Kinetic energy of a vortex patch near the boundary}\label{sec:lowerbound}
Next, we compute the energy of a vortex approximation near the boundary. We consider the vortex patch of height $\LL>0$ around $x_0\in \DD$, given by
\begin{equation}
\omega_{x_0,\LL}=\LL\ind_{B_{\sqrt{\frac{\sfm}{\LL}}}(x_0)},
\end{equation}
where we impose that $\LL$ satisfies
\begin{equation}\label{eq:Lconstraint}
\LL\ge \frac{4\sfm}{(1-|x_0|)^{2}},
\end{equation}
so that $\omega_{x_0,\LL}\in\cO_\LL$.
To estimate the kinetic energy of $\omega_{x_0,\LL}$, we use the explicit Green's function of the Laplace operator on the unit disk, so that
\begin{equation}
\psi_{x_0,\LL}(x)
=\frac{1}{2\pi}\int_{\DD} \log\frac{|x-y|}{|y||x-y_*|}\omega_{x_0,\LL}(y)\dd y=\frac{\LL}{2\pi}\int_{B_{\sqrt{\frac{\sfm}{ \LL}}}(x_0)} \log\frac{|x-y|}{|y||x-y_*|}\dd y,
\end{equation}
where $y_*=y/|y|^2$. Thus
\begin{equation}\label{eq:enercompt}
\sfE(\omega_{x_0,\LL})=-\frac12\int_\DD \psi_{x_0,\LL}(x)\omega_{x_0,M}(x)\dd x =\frac{\LL^2}{2\pi}\int_{B_{\sqrt{\frac{\sfm}{ \LL}}}(x_0)}\int_{B_{\sqrt{\frac{\sfm}{\LL}}}(x_0)}\log\frac{|y||x-y_*|}{|x-y|}\dd y\dd x.
\end{equation}
For $x,y\in B_{\sqrt{\frac{\sfm}{\LL}}}(x_0)$, we have the bound
\begin{equation}\label{eq:dist1}
|x-y|\le 2\sqrt{\frac{\sfm}{\LL}}.
\end{equation}
Using that $|y|>1/2$ and $y_*\notin \DD$, from \eqref{eq:Lconstraint} we deduce the bound
\begin{equation}\label{eq:dist2}
|y||x-y_*|\ge \frac12 \left[1-|x_0|-\sqrt{\frac{\sfm}{ \LL}}\,\right]=\frac12 \left[1-|x_0|-\frac{1-|x_0|}{2}\,\right]\ge \frac14 (1-|x_0|).
\end{equation}
Plugging \eqref{eq:dist1}-\eqref{eq:dist2} into \eqref{eq:enercompt} we obtain the bound
\begin{equation}\label{eq:ener1low}
\sfE(\omega_{x_0,\LL})\geq\frac{\pi\sfm^2}{2}\log \left(\sqrt{\frac{\LL}{\sfm}}\frac{(1-|x_0|)}{8}\right).
\end{equation}
Computing explicitly the angular momentum, we find
\begin{equation}
\sfA(\omega_{x_0,\LL})
= -\frac12\int_{\DD}(1-|x|^2)\omega_{x_0,\LL}(x)\dd x 
= -\frac{\pi\sfm}2\left(1 - |x_0|^2 -
\frac{\sfm}{2 \LL} \right)\ge -\pi\sfm(1-|x_0|).
\end{equation}
Hence, from \eqref{eq:ener1low} we have the bound
\begin{equation}\label{eq:energynonradial}
\sfE(\omega_{x_0,\LL})\geq\frac{\pi \sfm^2}{2}\log \left(\frac{\sqrt{\LL}\,|\sfA(\omega_{x_0,\LL})|}{8\pi\sqrt{m^3}}\right),
\end{equation}
as long as  
\begin{equation}\label{eq:LintermsofA}
L\ge      \frac{4\pi^2\sfm^3}{|\sfA(\omega_{x_0,\LL})|^2}
\end{equation}
to satisfy \eqref{eq:Lconstraint}.

\subsection{Proof of Theorem \ref{thm:nonrad}}
The bounds \eqref{eq:rad} and\eqref{eq:nonrad} are included in the sections above. Given an angular momentum $\sfa$, we consider the height
\begin{equation}\label{eq:heightnonradial}
L= Q^2 \pi^2\frac{\sfm^3}{|\sfa|^2 },   
\end{equation}
with $Q>2$ to be chosen below. By \eqref{eq:energynonradial} and \eqref{eq:heightnonradial}, we have the bound
\begin{equation}
\sup \left\{\sfE(\omega) : \omega\in\cO_\LL,\; \sfA(\omega)=\sfa\right\}\ge \frac{\pi\sfm^2}{2}\log \left(\frac{Q}{8}\right)
\end{equation}
which is independent of $\sfa$.
We pick 
$$
Q=8\e^{\frac{2}{\pi\sfm^2}},
$$
which implies
\begin{equation}
\sup \left\{\sfE(\omega) : \omega\in\cO_\LL,\; \sfA(\omega)=\sfa\right\}\ge 1.
\end{equation}
For radial functions, we use \eqref{eq:enerupbound} to get the bound
$$
\sup \left\{\sfE(\omega):\omega\in\cO_\LL,\; \sfA(\omega)=\sfa,\ \omega\ \mbox{is radial}\right\}
    \le C\left(\sfm |\sfa|+|\sfa|^2\log\left( Q^2 \pi^2 \frac{\sfm^3}{|\sfa|^3 }\right)\right).
$$
So to finish the proof we need to pick $\sfa_*<0$ close enough to zero depending only on $\sfm$ such that for any $\sfa\in(\sfa_*,0)$, we have
$$
C\left(\sfm |\sfa|+|\sfa|^2\log\left(Q^2\pi^2  \frac{\sfm^3}{|\sfa|^3 }\right)\right)<1,
$$
which implies the desired inequality.

\section{Stability of Onsager solutions with negative inverse temperature}\label{sec:Onsager}
Equations of Liouville type such as \eqref{eq:ELpsibeta} arise the classical setting of mean-field limits of the canonical Gibbs measure associated to a system of point vortices. We 
state below an important result from \cite{CLMP95}, for Onsager solutions, namely solutions to the mean-field equation \eqref{eq:MFequation}.

\begin{theorem}[\cite{CLMP95}*{Section 5}]\label{thm:caglioti}
    Let $\beta\in (-8\pi,\infty)$. Onsager solutions
\begin{equation}\label{eq:omegabeta1}
\omega_\beta(r)=\frac{1-A(\beta)}{\pi}\frac{1}{(1-A(\beta)r^2)^2}\qquad\mbox{with}\qquad A(\beta)=\frac{\beta}{8\pi+\beta}.    
\end{equation}
are the unique maximizer of
\begin{equation*}
        \sfF_\beta[\omega]=\sfS(\omega)-\beta\sfE(\omega)=- \int_{\DD}\omega\log\omega\;\dd x-\frac{\beta}{2} \int_{\DD}|\nabla \Delta^{-1}\omega|^2\;\dd x,
\end{equation*}
over the set 
\begin{equation*}
\cP=\left\{\omega\in L^1: \omega\geq 0, \  \int_\DD \omega(x)\dd x =1, \ \int_{\DD} \omega(x)\log\omega(x) \dd x <\infty\right\}. 
\end{equation*}
 Moreover, we have the convergence 
    $$
    \lim_{\beta\to -8\pi^{-}}\omega_\beta\to\delta_0,
    $$
weakly in the sense of measures.
\end{theorem}
The purpose of this section is to prove Theorem \ref{thm:ArnoldStability}. As mentioned already, the ideas related to (quantitative) rearrangement inequalities and elliptic equations
from \cites{Talenti76, amato2023talenti,CianchiEspositoFuscoTrombetti+2008+153+189,cianchi2008strengthened} are here revisited in the case of the disk $\DD$ and vorticities 
$\omega$ statisfying an $L^\infty$ bound. The key result for us is the following stability result with respect of the $H^{-1}$ norm and its radially decreasing rearrangement.

\begin{lemma}\label{lemma:quantitativeTalenti}
Consider a positive vorticity distribution $\omega\in L^\infty$ 
such that $\int_\DD\omega\;\dd x=\sfm>0$. 
Then, for every $\eps>0$ there exists $\delta>0$ such that if
$$
\sfE[\omega]-\sfE[\omega^\sharp]< \delta,
$$
then there exists $x_*\in \RR^2$ such that $|x_*|\le \eps$ and  
$$
\|\omega-\omega^\sharp(\cdot-x_*)\|_{L^1(\RR^2)}< \eps,
$$
where
$\omega$ and $\omega^\sharp$ are extended by zero  outside the disk.
\end{lemma}
The proof is deferred until after the proof of Arnold's stability.

\subsection{Proof of Theorem~\ref{thm:ArnoldStability}} 

We now proceed with the proof of the main result in this section.

\begin{proof}[Proof of Theorem \ref{thm:ArnoldStability}]
 Throughout the proof we take $\|\omega^{in}-\omega_\beta\|_{L^2}<\delta$ progressively smaller, and we keep changing $\eps$ accordingly and without renaming it. 
We will also omit the dependence on $t$ of the solution $\omega=\omega(t)$, as the proof is carried for any arbitrary $t\geq0$. We proceed in several steps. To simplify the notation, we first assume that 
\begin{equation}\label{ass:mass}
\int_\DD \omega_{in}(x)\;\dd x=\sfm=1    
\end{equation}
so that we have unit mass, and in the last step we generalize to the case $\sfm\ne 1$.

\vspace{0.3cm}

\noindent \textit{Step 0.} We show that for any $\eps>0$ small enough, we can pick $\|\omega_{\beta}-\omega^{in}\|_{L^2}<\delta$ small enough the corresponding Euler solution $\omega=\omega(t)$ is such that
    \begin{equation}\label{eq:energyineq}
         0\le \sfF_{\beta}[\omega_{\beta}]-\sfF_\beta[\omega(t)]< \eps\qquad\forall t\in[0,\infty).
    \end{equation}

\vspace{0.3cm}

\noindent\textit{Proof of Step 0.} We notice that $\sfF_\beta$ is continuous with respect to the $L^2$ norm. For what concerns the kinetic energy part, by the Poincaré inequality there exists $C_0>0$ such that
    $$
        |\sfE[\omega]-\sfE[\omega']|\le C_0\|\omega-\omega'\|_{L^2}, \qquad \forall \omega,\omega'\in L^\infty\cap \cP.
    $$
    For the Boltzmann entropy part, we notice that
    \begin{equation}\label{eq:unifint}
        |\omega\log\omega|\lesssim 1+|\omega|^2
    \end{equation}
    Thus if $\omega_n\to\omega$ in $L^2$, then up to subsequences $\omega_n\log \omega_n\to \omega\log \omega$ almost everywhere and 
    \eqref{eq:unifint} implies uniform integrability. Thus, Vitali convergence theorem implies that $\sfS[\omega_n]\to \sfS[\omega]$.
    
    Hence, given $\eps>0$, there exists $\delta=\delta(\eps)$ such that if $\|\omega_{\beta}-\omega^{in}\|_{L^2}<\delta$, then
    $$
        \left|\sfF_\beta[\omega_{\beta}]-\sfF_\beta[\omega^{in}]\right|< \eps.
    $$
    Using assumption \eqref{ass:mass}, that $\omega_{\beta}$ is the unique maximizer over $\mathcal{P}$ of $\sfF_\beta$, and that the mass and free energy $\sfF_\beta$ are conserved along the evolution of the Euler equations, we obtain the desired inequality \eqref{eq:energyineq}

\vspace{0.3cm}

\noindent   \textit{Step 1.} We can choose $\|\omega^{in}-\omega_\beta\|_{L^2}<\delta$, such that
\begin{equation}\label{eq:ineq2}
        \|\omega(t,\cdot)-\omega^\sharp(t,\cdot-x_*)\|_{L^2(\RR^2)}<\eps\qquad\forall t\in[0,\infty),
\end{equation}
where $x_*\in\RR$ satisfies $|x_*|<\eps$, and both functions are extended to $\RR^2$ by zero outside the disk.

\vspace{0.3cm}

\noindent\textit{Proof of Step 1.}
Using \eqref{eq:energyineq}, we have
    \begin{equation}\label{eq:auxfinal}
        0\le \underbrace{\sfF_\beta[\omega_\beta]-\sfF_\beta[\omega^\sharp(t)]}_{\ge 0}+\underbrace{\sfF_\beta[\omega^\sharp(t)]-\sfF_\beta[\omega(t)] 
    }_{\ge 0}< \eps\qquad\forall t\in[0,\infty),
    \end{equation}
    where the positivity of each term follows from the fact that $\omega_\beta$ is the optimizer of $\sfF_\beta$ over $\mathcal{P}$, and that $\sfF_\beta[\omega^\sharp(t)]\ge \sfF_\beta[\omega(t)]$ for $\beta<0$, in view of \eqref{eq:rearr}. Hence, using that $\sfS[\omega^\sharp]=\sfS[\omega]$ we can conclude that
    $$
        0\le \sfE[\omega^\sharp(t)]-\sfE[\omega(t)]< \eps.
    $$
    Up to notation, the conclusion of \eqref{eq:ineq2} follows from the quantitative Talenti's inequality in Lemma~\ref{lemma:quantitativeTalenti}. To get the stability in $L^2$ we just need to interpolate the above bound with the bounds in $L^\infty$.

    \vspace{0.3cm}

\noindent     \textit{Step 2.} 
    We consider the Brenier map $T:[0,1]\to [0,1]$ such that $T(0)=0$, $T(1)=1$ and 
    \begin{equation}\label{eq:Brenier2}
        \int_0^{r}\omega_\beta(s)s\,\dd s=\int_0^{T(r)}\omega^\sharp(s)s\,\dd s,\qquad \text{for any } r\in[0,1].
    \end{equation}
    For every $\eps>0$, we can pick $\|\omega_\beta-\omega^{in}\|_{L^2}<\delta$ small enough, such that
    \begin{equation}\label{eq:transportbounds}
        \int_0^1 |\omega_\beta(r)-\omega^\sharp(T(r))|^2\;r\dd r<\eps\qquad\mbox{and}\qquad \int_0^1 H\left(\frac{TT'}{r}\right)\;r\dd r<\eps,       
    \end{equation}
    where $H(u)=-\log u+u-1$.
    
    \vspace{0.3cm}
 \noindent \textit{Proof of Step 2.}    By \eqref{eq:auxfinal} we have that
    $$
        0\le \sfF_\beta[\omega_\beta]-\sfF_\beta[\omega^\sharp(t)]<\eps.
    $$
    Applying Lemma~\ref{lem:last}, we have that
    $$
        0\le  -2 \omega_\beta(1) \int_{0}^1\log\left(\frac{TT'}{r}\right) r\dd r \le \sfF_\beta[\omega_\beta]-\sfF_\beta[\omega^\sharp]<\eps.
    $$
The first inequality follows from an application of Jensen's inequality as in \eqref{eq:Jensen}.

Next, we apply a quantitative version of Jensen's inequality. Given $G(\cdot)=-\log(\cdot)$, considering the random variable $X=TT'/r$, and the probability measure $ 2r\;\dd r$ in $[0,1]$, we have 
$$
\mathbb{E} X= 2\int_0^1 TT'\;\dd r=T^2(1)-T^2(0)=1.
$$
Hence, for the function $G(\cdot)=-\log(\cdot)$, we have
$$
    \mathbb{E} G(X)=\mathbb{E}\left[ G(X)-G(\mathbb{E} X)-G'(\mathbb{E} X)(X-\mathbb{E} X)\right].
$$
We define the convex function
$$
H(x):=-\log x+x-1,
$$
which satisfies the inequality
\begin{equation}\label{eq:ineqH}
\frac14 \min\{|x-1|^2,|x-1|\}\le H(x), \qquad\forall x>0. 
\end{equation}
Using the observations above, we have
$$
 2\int_0^1 H\left(\frac{TT'}{r}\right) r\;\dd r\le -2\int_{0}^1\log\left(\frac{TT'}{r}\right) r\dd r <\eps.
$$
Differentiating \eqref{eq:Brenier2}, we have
\begin{equation}\label{eq:TT'}
    \frac{T(r)T'(r)}{r}=\frac{\omega_\beta(r)}{\omega^\sharp(T(r))},
\end{equation}
Applying \eqref{eq:ineqH} and using that $ \omega^\sharp,\omega_\beta\in L^\infty$, we can pick $\tilde{a}=\tilde{a}(\|\omega^\sharp\|_{L^\infty},\|\omega_\beta\|_{L^\infty})>0$ small enough such that
$$
\tilde{a}\left|\omega_\beta(r)-\omega^\sharp(T(r))\right|^2 \le \frac14 \min\left\{\left|\frac{\omega_\beta(r)}{\omega^\sharp(T(r))}-1\right|^2,\left|\frac{\omega_\beta(r)}{\omega^\sharp(T(r))}-1\right|\right\} \le H\left(\frac{\omega_\beta(r)}{\omega^\sharp(T(r))}\right).
$$
Therefore, up to relabelling $\eps$ we have
$$
\int_0^1\left|\omega_\beta(r)-\omega^\sharp(T(r))\right|^2r\;\dd r\le \eps.
$$

\vspace{0.3cm}

\noindent\textit{Step 3.} We exhibit a control on how far  the Brenier map is from the identity. That is to say, we can choose $\|\omega_\beta-\omega^{in}\|_{L^2}<\delta$ small enough, such that
\begin{equation}\label{eq:Tid}
|T(s)-s|<\eps,\qquad\forall s\in[0,1].    
\end{equation}

 \noindent \textit{Proof of Step 3.} By \eqref{eq:transportbounds}, for $\delta$ small enough we have the inequality
    $$
 \int_0^s H\left(\frac{TT'}{r}\right) r\;\dd r\le  \int_0^1H\left(\frac{TT'}{r}\right) r\;\dd r<\eps.
    $$
Applying Jensen's inequality, we have
$$
  \frac{s^2}{2}H\left(\frac{T^2(s)}{s^2}\right)=\frac{s^2}{2}H\left(\frac{2}{s^2}\int_0^s\frac{TT'}{r}r\;\dd r\right)\le    \int_0^s H\left(\frac{TT'}{r}\right) r\;\dd r<\eps.
$$
Using \eqref{eq:ineqH}, we have
$$
\frac14 s^2 \min\left\{\left|\frac{T^2(s)}{s^2}-1\right|^2,\left|\frac{T^2(s)}{s^2}-1\right|\right\}\le\frac{s^2}{2}H\left(\frac{T^2(s)}{s^2}\right).
$$
Using the mass constraint and the $L^\infty$ bound on $\omega$ we can take $r>0$ small enough depending on $\|\omega^{in}\|_{L^\infty}$   such that 
\begin{equation} 
   \omega^\sharp(x)\ge \frac{1-\|\omega^{in}\|_{L^\infty}|x|^2}{1-|x|^2}\ge \frac{1}{2},  \qquad\forall |x|<r.
\end{equation}
Hence, using \eqref{eq:Brenier2}, we have that for $r$ small enough
$$
C^{-1} T(r)\le \int_0^{T(r)}\omega^\sharp(s)s\,\dd s= \int_0^{r}\omega_\beta(s)s\,\dd s\le C r,
$$
which implies that uniformly 
$$
\sup_{r\in[0,1]}\left|\frac{T^2(r)}{r^2}-1\right|<C(\|\omega^\sharp\|_{L^\infty}).
$$
Hence, up to relabeling $\eps$, we have
$$
  |T(s)-s|^2\left|\frac{T(s)}{s}+1\right|^2=s^2\left|\frac{T^2(s)}{s^2}-1\right|^2 < \eps,
$$
and the conclusion follows.

\vspace{0.3cm}

\noindent \textit{Step 4.} For every $\eps>0$, we can pick $\|\omega_\beta-\omega^{in}\|_{L^2}<\delta$ small enough, such that
    \begin{equation}\label{eq:boundI}
        \int_0^1 |\omega_\beta(r)-\omega^\sharp(r)|\;r\dd r<\eps.      
    \end{equation} 

\noindent \textit{Proof of Step 4.}
Changing variables and applying the triangle inequality, we notice that
\begin{align*}
\int_0^1 |\omega_\beta(r)-\omega^\sharp(r)|\;r\dd r
&=\int_0^1 |\omega_\beta(T(r))-\omega^\sharp(T(r))|\; T(r)T'(r) \dd r\\
&\le\underbrace{\int_0^1 |\omega_\beta(T(r))-\omega_\beta(r)|\;T(r)T'(r)\dd r}_{I}+\underbrace{\int_0^1|\omega_\beta(r)-\omega^\sharp(T(r))|\; T(r)T'(r) \dd r}_{II}\nonumber.
\end{align*}
Applying \eqref{eq:Tid} and the smoothness of $\omega_\beta$, we get the bound
$$
I\le \|\de_r\omega_\beta\|_{L^\infty}\|T-r\|_{L^\infty} \int_0^1 T(r)T'(r)\;\dd r<\|\de_r\omega_\beta\|_{L^\infty}\eps.
$$
We manipulate the second term to get
\begin{align*}
    II&\le  \int_0^1|\omega_\beta(r)-\omega^\sharp(T(r))|\; r \dd r+\int_0^1|\omega_\beta(r)-\omega^\sharp(T(r))|\; (TT'-r) \dd r \le \sqrt{\frac{\eps}{2}} +\eps,
\end{align*}
where we have used Cauchy-Schwarz and  \eqref{eq:transportbounds}. Up to renaming $\eps$, \eqref{eq:boundI} now follows.

\vspace{0.3cm}

\noindent \textit{Step 5.} We now conclude the proof of the theorem, under the assumption the unit mass assumption \eqref{ass:mass}.

\vspace{0.3cm}

\noindent \textit{Proof of Step 5.}
We combine \eqref{eq:ineq2} and \eqref{eq:boundI}, to obtain that  
$$
\|\omega_\beta(\cdot-x_*)-\omega(t)\|_{L^1(\RR^2)}< 2\eps,\qquad \forall t\geq0,
$$
for some $x_*\in\RR^2$ such that $|x_*|\leq\eps$, where we have extended the functions by zero outside the disk $\DD$. Using the continuity of the $L^1$ norm over translations for $\omega_\beta$, we can conclude that, up to renaming $\eps$,
$$
\|\omega_\beta-\omega(t)\|_{L^1(\RR^2)}<\eps.
$$
To get the stability in $L^2$ we just need to interpolate the above bound with the bounds in $L^\infty$.

\vspace{0.3cm}

\noindent  
\textit{Step 6.} We conclude conclude the proof of the theorem, without the unit mass assumption \eqref{ass:mass}.

\vspace{0.3cm}

\noindent\textit{Proof of Step 6.} We start by showing that the maximizer of the problem 
\begin{equation}\label{eq:probm}
        \max_{\omega\in \cP_\sfm}\sfF_\beta[\omega]=\sfS(\omega)-\beta\sfE(\omega)=- \int_{\DD}\omega\log\omega\;\dd x-\frac{\beta}{2} \int_{\DD}|\nabla \Delta^{-1}\omega|^2\;\dd x,
\end{equation}
where
\begin{equation*}
\cP_\sfm=\left\{\omega\in L^1: \omega\geq 0, \  \int_\DD \omega(x)\dd x =\sfm, \ \int_{\DD} \omega(x)\log\omega(x) \dd x <\infty\right\},
\end{equation*}
is continuous with respect to the mass parameter $\sfm$ in any $L^p$ with $p\in[1,\infty)$, as the long the parameter $\sfm\beta>-8\pi$.
By re-scaling, we can relate
$$
\max_{\omega\in \cP_\sfm}- \int_{\DD}\omega\log\omega\;\dd x-\frac{\beta}{2} \int_{\DD}|\nabla \Delta^{-1}\omega|^2\;\dd x=\max_{\omega\in \cP_1}-\sfm \int_{\DD}\omega\log\omega\;\dd x-\frac{\beta\sfm^2}{2} \int_{\DD}|\nabla \Delta^{-1}\omega|^2\;\dd x-\sfm\log\sfm
$$
Hence, the for general $\sfm$, we get that the maximizer of \eqref{eq:probm} is given by
$$
\omega_*=\sfm \omega_{\sfm \beta},
$$
which by \eqref{eq:omegabeta} is continuous in any topology as long as it does not blow up $\sfm\beta>-8\pi$.

Using the continuity of the mass with respect to the $L^2$ norm, we can pick $\delta$ small enough so that $\|\omega_\beta-\omega_{in}\|_{L^2}<\delta$ implies
\begin{equation}\label{eq:111}
    \|\sfm \omega_{\sfm \beta}-\omega_\beta\|_{L^2}<\eps\qquad\mbox{and}\qquad\|\sfm\omega_{\sfm\beta}-\omega_{in}\|_{L^2}<\eps
\end{equation}
where
$$
\sfm=\int_{\DD}\omega_{in}(x)\;\dd x.
$$
To conclude the proof, we need to repeat \textit{Steps 0-5} replacing the role of $\omega_\beta$ by $\sfm\omega_{\sfm\beta}$. We conclude that we can pick $\delta$ small enough $\|\omega_\beta-\omega_{in}\|_{L^2}<\delta$ such that $\|\sfm\omega_{\sfm\beta}-\omega(t)\|_{L^2}<\eps$ for all $t>0$, and the result follows, up to relabeling $\eps$, by \eqref{eq:111}.

\end{proof}
\subsection{Stability of rearrangements}

We now proceed to prove Lemma~\ref{lemma:quantitativeTalenti}, which constitute the crucial step in the proof of Theorem \ref{thm:ArnoldStability} above.

\begin{proof}[Proof of Lemma~\ref{lemma:quantitativeTalenti}]
We consider (up to signs!) the associated stream functions to $\omega$ and its rearrangement $\omega^\#$
\begin{equation}\label{eq:ellipticprob}
    \begin{cases}
            -\Delta  \phi=\omega,\quad & \mbox{in}\; \DD,\\
            \phi=0, & \mbox{on}\;\partial\DD,
        \end{cases}
        \qquad
        \begin{cases}
            -\Delta \bar \phi=\omega^\sharp,\quad  & \mbox{in}\; \DD,\\
            \bar\phi=0, & \mbox{on}\;\partial\DD.
        \end{cases}
\end{equation}
A celebrate theorem of Talenti \cite{Talenti76}*{Theorem 1} states   that  
$$
\phi^\sharp(x)\le \bar\phi(x),\qquad\forall x\in\DD,
$$
and  
$$
\sfE[\omega]=\frac12\|\nabla \phi\|^2_{L^2 }\leq\frac12\|\nabla \bar\phi\|^2_{L^2 }=\sfE[\omega^\sharp].
$$
To prove the lemma, we again proceed in steps.

\vspace{0.3cm}

\noindent\textit{Step 1.} We show that under our hypothesis, there exists $C(\|\omega\|_{L^\infty},\sfm)>0$ such that
\begin{equation}\label{eq:bound1}
    \|\bar{\phi} -\phi^{\sharp} \|_{L^\infty}\le  C\left(\sfE[\omega^\sharp]-\sfE[\omega]\right).
\end{equation}

\vspace{0.3cm}

\noindent\textit{Proof of Step 1.}
For $h\geq 0$, we consider the distribution function
$$
u(h)=|\{x\in\DD\;:\;\phi(x)>h\}|,
$$
whose derivative is
$$
u'(h)=-\int_{\partial [\phi>h]}\frac{1}{|\nabla \phi|}\dd\mathcal{H}^1.
$$
Considering the perimeter of the level sets and the isoperimetric inequality we obtain
$$
2\pi^{\frac12} u(h)^{\frac12}\le \mathrm{Per}([\phi>h])=\int_{\partial [\phi>h]} \dd\mathcal{H}^1\le \left(-u'(h)\int_{\partial [\phi>h]}|\nabla \phi| \dd\mathcal{H}^1\right)^{\frac12}.
$$
Next, we use the first equation in \eqref{eq:ellipticprob} to compute the last integral
$$
\int_{\partial [\phi>h]}|\nabla \phi| \dd\mathcal{H}^1=\int_{[\phi>h]} -\Delta\phi\;\dd x 
=\int_{[\phi>h]}\omega\;\dd x\le \int_0^{\left(\frac{u(h)}{\pi}\right)^{\frac12}} \omega^\sharp(s)s\;\dd s.
$$
Putting the last two equations together, we get the inequality
\begin{equation}\label{eq:der}
    4\pi u(h)\le -u'(h)\int_0^{\left(\frac{u(h)}{\pi}\right)^{\frac12}}\omega^\sharp(s)s\;\dd s.
\end{equation}
Noting that for the rearranged vorticity $\omega^\sharp$ all the inequalities are in fact equalities, we obtain that the distribution $v(h)=|\{x\in\DD\;:\;\bar\phi(x)>h\}|$ function satisfies
\begin{equation}\label{eq:der1}
4\pi v(h)= -v'(h)\int_0^{\left(\frac{v(h)}{\pi}\right)^{\frac12}}\omega^\sharp(s)s\;\dd s.
\end{equation}
Using the boundary condition we have that 
$$
u(0)=v(0)=\pi,
$$
and hence we can use the derivative equations \eqref{eq:der} and \eqref{eq:der1} to conclude that 
$$
u(h)\le v(h),\qquad\forall h\ge 0.
$$
Using the inequality above with
\begin{equation}\label{eq:assoc}
u(\phi^{\sharp}(x))=\pi|x|^2=v(\bar{\phi}(x))\qquad\mbox{and}\qquad v'(h)<0,    
\end{equation}
we get Talenti's inequality
\begin{equation}\label{eq:talenti}
   \phi^{\sharp}(x)\le \bar{\phi}(x). 
\end{equation}
Using \eqref{eq:assoc}, with \eqref{eq:der} and \eqref{eq:der1}, we obtain that 
$$
\de_r\phi^\sharp(r)\ge \de_r\bar\phi(r), \qquad\forall r\in[0,1],
$$
where we have abused notation and considered $\phi^\sharp$ and $\bar\phi$ with respect to the radial variable. This implies
$$
\max_{x\in\DD} |\bar{\phi}(x)-\phi^{\sharp}(x)|=\bar{\phi}(0)-\phi^{\sharp}(0).
$$
Since $\omega,\, \omega^\sharp \in L^\infty$, the corresponding streamfunctions  $\phi,\,\bar\phi$ are Lipschitz-continuous. As radial rearrange\-ments are contractive in the Lipschitz norm, we have
$$
\|\phi^\sharp\|_{W^{1,\infty}}\le \|\phi\|_{W^{1,\infty}}\lesssim \|\omega\|_{L^\infty}.
$$
Hence, there exists $r>0$ small enough depending only on $\|\omega\|_{L^\infty}$ such that 
\begin{equation}\label{eq:auxa}
|\bar{\phi}(x)-\phi^{\sharp}(x)|\ge \frac{1}{2}\|\bar{\phi}-\phi^{\sharp}\|_{L^\infty},\qquad\forall |x|<r.    
\end{equation}
Using the mass constraint and the $L^\infty$ bound on $\omega$ we can take $r>0$ small enough depending on $\|\omega\|_{L^\infty}$ and $\sfm$ such that 
\begin{equation}\label{eq:auxb}
   \omega^\sharp(x)\ge \frac{\sfm-\|\omega\|_{L^\infty}|x|^2}{1-|x|^2}\ge \frac{\sfm}{2},  \qquad\forall |x|<r.
\end{equation}

Now we are ready to show \eqref{eq:bound1}. By the the Hardy-Littlewood inequality,
\begin{align*}
     \sfE[\omega^\sharp]-\sfE[\omega]
     &=\frac12\int_{\DD} \bar\phi\omega^\sharp-\phi\omega\;\dd x\\
     &=\frac12\int_{\DD} \phi^\sharp\omega^\sharp-\phi\omega+(\bar\phi-\phi^\sharp)\omega^\sharp\;\dd x\\
     &\ge \frac12\int_{\DD}(\bar\phi-\phi^\sharp)\omega^\sharp\;\dd x\\
     &\ge \frac12\int_{B_r}(\bar\phi-\phi^\sharp)\omega^\sharp\;\dd x\\
     &\ge  c(\|\omega\|_{L^\infty},\sfm)\|\bar{\phi}-\phi^{\sharp}\|_{L^\infty},
\end{align*}
where we used \eqref{eq:talenti}, \eqref{eq:auxa} and \eqref{eq:auxb}.

\vspace{0.3cm}

\noindent\textit{Step 2.} For any $\eps>0$, there exists $\delta>0$ such that if $\|\bar\phi-\phi^\sharp\|_{L^\infty}<\delta$, there exists $x_*\in\RR^2$ such that $|x_*|<\eps$ and 
    $$
        \|\omega-\omega^\sharp(\cdot-x_*)\|_{L^1(\RR^2)}<\eps
    $$
    where the functions $\omega$ and $\omega^\sharp$ are extended by zero  outside the disk.

\vspace{0.3cm}
\noindent\textit{Proof of Step 2.} 
We choose $x_*$ to be the optimizer of 
    $$
    \|\phi-\phi^\sharp(\cdot-x_*)\|_{L^2(\RR^2)}=\inf_{x_0\in\RR^2}\|\phi-\phi^\sharp(\cdot-x_0)\|_{L^2(\RR^2)}.
    $$
    Applying  \cite{amato2023talenti}*{Section 5, eqn. (82)}, we have that there exists a $C>0$ such that
    $$
        C^{-1}\min(|x_*|,1/2)\le|\DD\triangle(\DD+x_*)|\le C\|\bar\phi-\phi^\sharp\|^{1/4}_{L^\infty}
    $$
    The proof of \cite{amato2023talenti}*{Theorem 1.4} shows exactly that we can pick $\|\bar\phi-\phi^\sharp\|_{L^\infty}<\delta$ small enough such that
    $$
        \inf_{x_0\in\RR^2} \|\omega-\omega^\sharp(\cdot-x_0)\|_{L^1} \le \|\omega-\omega^\sharp(\cdot-x_*)\|_{L^1}<\eps.
    $$
This concludes the proof.
\end{proof}

\appendix

\section{A min-max principle}\label{app:minmax}
We show a variation of the classical min-max principle, which can be found in \cite{ET}*{Chapter VI}.

\begin{proposition}\label{prop:minmax}
Let $A$ and $B$ be closed convex sets of a Banach space $(X_1,\|\cdot\|_1)$ and $(X_2,\|\cdot\|_2)$, and consider a proper functional $\sfL:A\times B \to \RR$. Assume the following.
\begin{enumerate}[label=(\alph*), ref=(\alph*)]
\item For every $\beta\in A$, the function $\omega \mapsto \sfL(\beta,\omega)$ is  weakly u.s.c.\label{a}.
\item For every $\omega\in B$, the function $\beta \mapsto \sfL(\beta,\omega)$ is convex and l.s.c.. \label{b}
\item The functional $\sfL$ is coercive in $\beta$. More specifically, there exists a function $g:[0,\infty)\to[0,\infty)$ such that $\lim_{u\to\infty}g(u)=\infty$ and for any $\beta\in A$ there exists $\omega\in B$ so that
$$
\sfL(\beta,\omega)\ge g(\|\beta\|_{1}).
$$\label{c}
\item The set $B$ is bounded, hence weakly compact.\label{d}
\item\label{item:minmaxunique} For every $\beta\in A$, the function $\omega \mapsto \sfL(\beta,\omega)$ has a unique maximizer $\omega_\beta$.\label{e}
\end{enumerate}
Then
\begin{equation}
\inf_{\beta\in A}\sup_{\omega\in B}\sfL(\beta,\omega)=\sup_{\omega\in B}\inf_{\beta\in A}\sfL(\beta,\omega).
\end{equation}
\end{proposition}

\begin{remark}\label{rem:app}
 The condition \ref{e} can be weakened to the following:
    For any $\beta_*\in A$ any two maximizers $\omega_1^*$ and $\omega_2^*$ satisfy
    $$
        \sfL(\beta,\omega_1^*)=\sfL(\beta,\omega_2^*),\qquad\mbox{for any other $\beta\in A$.}
    $$
    In the case of the Euler equation in a radial domain this means that min-max principle applies if we know that the maximizers are unique up to rigid rotations, which preserves the energy and the entropy.
\end{remark}
\begin{proof}[Proof of Proposition \ref{prop:minmax}]
First of all, we observe that
\begin{equation}
\sfL(\beta,\omega)\geq\inf_{\beta\in A}\sfL(\beta,\omega), \qquad  \forall \omega\in B,
\end{equation}
so that
\begin{equation}
\sup_{\omega\in B} \sfL(\beta,\omega)\geq\sup_{\omega\in B}\inf_{\beta\in A}\sfL(\beta,\omega),
\end{equation}
and thus
\begin{equation}
\inf_{\beta\in A}\sup_{\omega\in B} \sfL(\beta,\omega)\geq\sup_{\omega\in B}\inf_{\beta\in A}\sfL(\beta,\omega),
\end{equation}
so we only need to prove the reverse inequality.
Define 
$$
f(\beta):= \sfL(\beta,\omega_\beta)=\sup_{\omega\in B} \sfL(\beta,\omega),
$$ 
where $\omega_\beta\in B$ is assumed to be the unique maximizer from \ref{e}. The function $\beta\mapsto f(\beta)$ is convex and l.s.c., being the envelope of convex l.s.c. functions by assumption \ref{b}. Therefore by convexity and coercivity \ref{c} it attains its lower bound at some $\bar\beta \in A$, so that
\begin{equation}
f(\bar\beta)=\min_{\beta\in A} f(\beta)=\min_{\beta\in A} \max_{\omega\in B} \sfL(\beta,\omega)
\end{equation}
and
\begin{equation}
f(\bar\beta)\geq \sfL(\bar\beta,\omega), \qquad \forall \omega \in B.
\end{equation}
Now, by convexity \ref{b} for every $\beta\in A,\omega \in B$ and $t\in(0,1)$, we have
\begin{equation}
\sfL((1-t)\bar\beta+t\beta,\omega)\leq (1-t)\sfL(\bar\beta,\omega)+t\sfL(\beta,\omega)
\end{equation}
In particular, taking $\beta_t=(1-t)\bar\beta+t\beta$ we consider $\omega=\omega_{\beta_t}$ given by \ref{e}, we find
\begin{align}
f(\bar\beta)\leq f(\beta_t)= \sfL(\beta_t,\omega_{\beta_t})\leq (1-t)\sfL(\bar\beta,\omega_{\beta_t})+t\sfL(\beta,\omega_{\beta_t})\leq (1-t)f(\bar\beta)+t\sfL(\beta,\omega_{\beta_t}),
\end{align}
implying
\begin{align}\label{eq:2.7temam}
f(\bar\beta)\leq \sfL(\beta,\omega_{\beta_t}), \qquad \forall \beta\in A.
\end{align}
Now, by compactness \ref{d}, as $t\to 0$, $\omega_{\beta_t}$  converges weakly to some $\bar\omega\in B$, up to subsequences. Next, we claim that $\bar\omega=\omega_{\bar\beta}$. Indeed,
\begin{equation}
\sfL(\beta_t,\omega_{\beta_t})\geq \sfL(\beta_t,\omega), \qquad \forall \omega\in B,
\end{equation}
and from convexity \ref{b} we have 
\begin{equation}
(1-t)\sfL(\bar\beta,\omega_{\beta_t})+t\sfL(\beta,\omega_{\beta_t})\geq \sfL(\beta_t,\omega), \qquad \forall \omega\in B.
\end{equation}
Since $\sfL(\beta,\omega_{\beta_t})\leq f(\beta)<\infty$, we can use the semi-continuity \ref{a} and \ref{b} to pass to the limit as $t\to 0$ and obtain
\begin{equation}
\sfL(\bar\beta,\bar\omega)\geq \limsup_{t\to0} (1-t)\sfL(\bar\beta,\omega_{\beta_t})+t\sfL(\beta,\omega_{\beta_t})\geq \liminf_{t\to0} \sfL(\beta_t,\omega)\geq \sfL(\bar\beta,\omega), \qquad \forall \omega\in B,
\end{equation}
proving the claim that $\bar\omega=\omega_{\bar\beta}$, the unique maximizer by \ref{e}. We can now pass to the limit in \eqref{eq:2.7temam} using weak u.s.c. \ref{a} to get
\begin{align}\label{eq:2.8temam}
f(\bar\beta)\leq \sfL(\beta,\omega_{\bar\beta}), \qquad \forall \beta\in A.
\end{align}
Thus
\begin{equation}
\min_{\beta\in A} \max_{\omega\in B} \sfL(\beta,\omega)\leq \min_{\beta\in A} \sfL(\beta,\omega_{\bar\beta}) \leq \max_{\omega\in B}\min_{\beta\in A} \sfL(\beta,\omega),
\end{equation}
as we needed, and the proof is complete.
\end{proof}

\addtocontents{toc}{\protect\setcounter{tocdepth}{0}}

 \section*{Acknowledgments} 
The research of MCZ was partially supported by the Royal Society URF\textbackslash R1\textbackslash 191492 and EPSRC Horizon Europe Guarantee EP/X020886/1. The research of MGD was partially supported by NSF-DMS-2205937 and NSF-DMS RTG 1840314. The authors would also like to thank AIMS Senegal for their hospitality in the early stages of this project. We would like to thank J.A. Carrillo, E. Caglioti, T.D. Drivas and V. Šverák for illuminating discussions that helped improving this work.

\addtocontents{toc}{\protect\setcounter{tocdepth}{1}}

\bibliographystyle{alpha}
\bibliography{Biblio-Entropy2dEuler}

\end{document}